\documentclass[12pt, review,authoryear]{elsarticle} 
\usepackage{amssymb,amsmath}

\newenvironment{proof}{\noindent {\bf Proof }}
{\hfill $\square$ \vspace{0.25cm}}

\def\E{{\mathbb E}}
\def\P{{\mathbb P}}
\def\R{{\mathbb R}}
\def\Z{{\mathbb Z}}

\def\N{{\mathbb N}}

\newcommand{\leftexp}[2]{{\vphantom{#2}}^{#1}{#2}}
\newcommand{\G}{\ensuremath{\left}}    
\newcommand{\D}{\ensuremath{\right}}   
    
\newtheorem{example}{Example}[section]
\newtheorem{lemma}{Lemma}[section]

\newtheorem{obs}{Observation}[section]
\newtheorem{defi}{Definition}[section]
\newtheorem{theo}{Theorem}[section]
\newtheorem{prop}{Proposition}[section]
\newtheorem{coro}{Corollary}[section]

\begin{document}

\begin{frontmatter}
\title{Perfect simulation for locally continuous chains of infinite
  order}
\author{Sandro Gallo}
\ead{sandro@im.ufrj.br}
\address{Instituto de Matem\'atica, Universidade Federal de Rio
  de Janeiro, Brazil}
\author{Nancy L. Garcia\corref{cor2}}
\ead{nancy@ime.unicamp.br}
\address{Instituto de Matem\'atica, Estat\'istica e Computa\c
  c\~ao Cient\'ifica, Universidade Estadual de Campinas, Brazil}
\cortext[cor2]{ Corresponding  author}

 \date{October 05, 2012}


\begin{abstract}
  We establish sufficient conditions for perfect simulation of chains
  of infinite order on a countable alphabet. The new
  assumption, \emph{localized continuity}, is formalized with the help
  of the notion of \emph{context trees}, and includes the traditional
  continuous case, probabilistic context trees and discontinuous
  kernels. Since our assumptions are more refined than uniform
  continuity, our algorithms perfectly simulate continuous chains
  faster than the existing algorithms of the literature. We provide
  several illustrative examples.
\end{abstract}
\end{frontmatter}


\section{Introduction}

The objects of this paper are stationary stochastic chains of infinite
order on a countable alphabet.  {These chains are said to be
  compatible} with a set of transition probabilities (depending on an
unbounded portion of the past) if the later is a regular version of
the conditional expectation of the former. This reflects the idea that
chains of infinite order are usually determined by their conditional
probabilities with respect to the past.  Given a set of transition
probabilities (or, simply \emph{kernels} in the sequel), two natural
questions are {(1) \emph{existence:} does there exist a stationary
  chain compatible with it? And (2)\emph{uniqueness:} if yes, is it
  unique?  {A constructive way to answer positively
    these questions is to provide an algorithm based on the transition
    probabilities which converges a.s. and samples precisely from the
    stationary law of the process compatible with the given
    kernel. This is precisely the focus of
    this paper. }

Perfect simulation for chains of infinite order was first
    done by \cite{comets/fernandez/ferrari/2002} under the continuity
    assumption. They used the fact, observed earlier by
    \cite{kalikow/1990}, that under this assumption, the transition
    probability kernel can be decomposed as a countable mixture of
    Markov kernels. Then, \cite{gallo/2009} obtained a perfect
    simulation algorithm for chains compatible with a class of
    unbounded probabilistic context trees where each infinite size
    branch can be a discontinuity point. Both of them
   use an extended version of the so-called \emph{coupling from the
      past algorithm} (CFTP in the sequel) introduced by
     \cite{propp/wilson/1996} for Markov chains. 
     Recently, \cite{gallo/garcia/2010} proposed a combination between
     these algorithms to cover cases where the kernels are neither
     necessarily continuous nor necessarily probabilistic context
     trees. In the present paper we consider a broader class of
     kernels that includes all the above cases, in fact, all the results of
     the above cited works can be obtained as corollaries of
     the present work. 
     
     Other recent results in the area are the papers of
     \cite{garivier/2011} and \cite{desantis/piccioni/2012}.  The
     former introduced an elegant CFTP algorithm which works without
     the weak non-nullness assumption, designed \emph{\`a la}
     \cite{propp/wilson/1996}. Their work does not intend to exhibit
     explicit sufficient conditions for a CFTP to be feasible and has
     a more algorithmic-motivated approach. The later introduced an
     interesting framework, making use of an a priori knowledge about
     the histories, extracted from the auxiliary sequence of random
     variables used for the simulation. Their general conditions are
     not explicitly given on the kernel, difficulting the comparison
     with our method. Notice, however, that our result is a strict
     generalization of Theorem 4.1 in
     \cite{comets/fernandez/ferrari/2002} whereas the regime of slow
     continuity rate is not present in their paper. To be more
     transparent, we show that all the examples of
     \cite{desantis/piccioni/2012} satisfy our conditions when
     considering the weakly non-null cases.

   Let us emphasize also that discontinuities appear quite
   naturally. Section \ref{sec:examples1} present several examples. On
   the other hand, relaxing the continuity assumption has an interest
   not only from a mathematical point of view, but also from an
   applied point of view. Practitioners generally seek to build models
   which are as general as possible. From data, it is not possible to
   check the rate of decay of the dependence on the past, and
   therefore, we do not know if we have continuity.

One of the main concepts introduced in this paper is the notion of
{\it skeleton} related to a transition probability
kernel. It is the smallest context tree composed by the set of pasts
which have a continuity rate which converges slowly to zero, or even
which does not converge to zero (discontinuity points).  
{This concept is reminiscent} of the
concept of \emph{bad} pasts, meaning the set of discontinuous pasts
for a given two-sided specification, which appears in the framework of
almost-Gibbsianity in the statistical physics literature. We refer to
\cite{vanEnter/aernout/redig/verbitskiy/2008} for a discussion on the
subject.  Almost-Gibbs measures appear in several situations, for
example random walk in random scenery (see for example,
\cite{denHollander/steif/vanderWal/2005} and
\cite{denHollander/steif/2006}), projection of the Ising model on a
layer (\cite{maes/redig/vanMoffaert/leuven/1999}), intermittency
(\cite{maes/redig/takens/vanMoffaert/verbitski/2000}), projection of
Markov measures (\cite{chazottes/ugalde/2011}). From this point of
view,  our work exhibits a large class of
almost-Gibbs measures that can be perfectly simulated.

Our first main result, Theorem \ref{theo2}, deals with locally
continuous chains.  Local continuity corresponds to assume that there
exists a stopping time for the reversed-time process, beyond which the
decay of the dependence on the past occurs uniformly. Theorem
\ref{theo2} states that, if the localized-continuity rate decays fast
enough to zero, we can perfectly simulate the stationary chain by
CFTP. More precisely, according to this rate, we specify several
regimes for the tail distribution of the coalescence time.  Theorem
\ref{theo2.5} presents an interesting extension where we remove the
local continuity assumption.  This means that this later result deals
with chains such that no stopping time can tell whether or not the
past we consider is a continuity point for $P$. Here also, we give
explicit examples, motivating these two theorems.

{It is important to emphasize that these results not only enlarge the class of
  processes which can be perfectly simulated but also they can be
  interpreted as a method to ``speed up'' the perfect simulation
  algorithm proposed by \cite{comets/fernandez/ferrari/2002}. Assume
  that the kernel is such that their algorithm can be performed (that
  is, is continuous, with a sufficiently fast continuity rate), but
  with an infinite expected time, due to some pasts which slow down
  the continuity rate. Our method allows us to include these pasts
  into the set of infinite size contexts of the skeleton.
  Then, depending on the position of these branches (that is,
  depending on the \emph{form} of the skeleton), our
  results show that the perfect simulation might be done in a finite
  expected time.}

{Our perfect simulation algorithm for the
given kernel $P$ requires that the skeleton is itself
perfectly simulable. However, this apparent handicap
is easy to overcome. Sufficient conditions for perfect simulability
can be explicitly obtained for a wide class of skeleton context
trees.} Several examples are presented and used throughout the paper.

It is worth mentioning that \cite{foss/konstantopoulos/2004} showed
that the notion of perfect simulation based on a coupling from the
past is closely related to the almost sure existence of ``renovating
event''. However, the difficulty always lies in finding such event for
each specific problem. In the present case, the perfect simulation
scheme provides such renovating event and gives conditions for almost
sure occurrence in terms of the transition probability kernel.

\vspace{0.2cm}

The paper is organized as follows.  In Section
  \ref{sec:notation} we present the basic definitions, the notation,
  and we introduce the coupling from the past algorithm for perfect
  simulation in a generic way. Section \ref{sec:assumptions}
  introduces the more specific notions of local continuity and
  skeletons, that are fundamental for our approach to
  perfect simulation. Our first main result on perfect simulation
  (Theorem \ref{theo2}) is presented in Section \ref{sec:statements},
  together with the corollaries of existence, uniqueness and
  regeneration scheme which are directly inherited by the constructed
  stationary chain. Discussion of these results and explicit examples
  of application are given in Section \ref{sec:examples1}. The proof
  of Theorem \ref{theo2} is given in Section \ref{sec:prova2}. Section
  \ref{sec:extension} is dedicated to an extension of Theorem
  \ref{theo2}. We finish the paper with some concluding remarks in
  Section \ref{sec:conclu}.

\section{Notation and {basic definitions}}\label{sec:notation}

Let $A$ be a countable alphabet. Given two integers $m\leq n$, we
denote by $a_m^n$ the string $a_m \ldots a_n$ of symbols in $A$. For
any $m\leq n$, the length of the string $a_m^n$ is denoted by
$|a_m^n|$ and defined by $n-m+1$. We will often use the notation
$\emptyset$ which will stand for the empty string, having length
$|\emptyset|=0$. For any $n\in\mathbb{Z}$, we will use the convention
that $a_{n+1}^{n}=\emptyset$, and naturally $|a_{n+1}^{n}|=0$. Given
two strings $v$ and $v'$, we denote by $vv'$ the string of length $|v|
+ |v'| $ obtained by concatenating the two strings. If $v'=\emptyset$,
then $v\emptyset=\emptyset v=v$. The concatenation of strings is also
extended to the case where $v=\ldots a_{-2}a_{-1}$ is a semi-infinite
sequence of symbols. If $n\in\{1,2,\ldots\}$ and $v$ is a finite
string of symbols in $A$, $v^{n}=v\ldots v$ is the concatenation of
$n$ times the string $v$. In the case where $n=0$, $v^{0}$ is the
empty string $\emptyset$. Let
$$
A^{-\mathbb{N}}=A^{\{\ldots,-2,-1\}}\,\,\,\,\,\,\textrm{ and }\,\,\,\,\,\,\,  A^{\star} \,=\, \bigcup_{j=0}^{+\infty}\,A^{\{-j,\dots, -1\}}\, ,
$$
be, respectively, the set of all infinite strings of past symbols and the set of all finite strings of past symbols. 
The case $j=0$ corresponds to the empty string $\emptyset$. Finally, we denote by    $\underline{a}=\ldots a_{-2}a_{-1}$ the elements of $A^{-\mathbb{N}}$.

Along this paper, we will often use the letters $u$, $v$ and $w$ for (finite
or infinite) strings of symbols of $A$, and the letters $i$, $j$, $k$,
$l$, $m$ and $n$ for integers.

\begin{defi}A \emph{transition probability kernel} (or simply \emph{kernel} in the
sequel) on a countable alphabet $A$ is a function
\begin{equation}
\begin{array}{cccc}
P:&A\times A^{-\mathbb{N}}&\rightarrow& [0,1]\\
&(a,\underline{a})&\mapsto&P(a|\underline{a})
\end{array}
\end{equation}
such that $\sum_{a\in A}P(a|\underline{a})=1$ for any $\underline{a}\in A^{-\mathbb{N}}$.
\end{defi}

For any $a\in A$, we define 
\[
\alpha(a):=\inf_{\underline{z}}P(a|\underline{z})\,\,\,\textrm{and} \,\,\,\alpha_{-1}:=\sum_{a\in A}\alpha(a).
\]

\begin{defi}
We say that the kernel $P$ is  \emph{weakly non-null} if 
$\alpha_{-1}>0$.
\end{defi}

Notice that, a given kernel $P$ is Markovian of order $k$ if 
$P(a|\underline{a})=P(a|\underline{b})$ for any $\underline{a}$ and $\underline{b}$ such that
$a_{-k}^{-1}=b_{-k}^{-1}$.

 A given kernel $P$ is
\emph{continuous} (with respect to the product topology) at some point
$\underline{a}$ if $P(a|a_{-k}^{-1}\underline{z})\rightarrow P(a|\underline{a})$
whenever $k$ diverges, for any $\underline{z}$. Continuous kernels are a natural
extension of Markov kernels. 

In this work we will need an equivalent definition of continuity.

\begin{defi}
We say that an infinite sequence of past symbols 
$\underline{b}$ is a \emph{continuity point} (or past) for a given kernel $P$ if the sequence $\{\omega_{k}(\underline{b})\}_{k\ge1}$ defined by
\begin{equation}\label{eq:pointcontinuity}
\omega_{k}(\underline{b}):=\sum_{b\in A}\inf_{\underline{z}}P(b|b_{-k}^{-1}\,\underline{z})\,,\,\,k\ge1,
\end{equation}
converges to $1$, and a discontinuity point otherwise.
 We say that $P$ is (uniformly) \emph{continuous} if the sequence $\{\omega_{k}\}_{k\ge1}$ defined by
\begin{equation}\label{eq:continuity}
\omega_{k}:=\inf_{b_{-k}^{-1}}\sum_{b\in A}\inf_{\underline{z}}P(b|b_{-k}^{-1}\,\underline{z})\,,\,\, k\ge1,
\end{equation}
converges to $1$, and discontinuous otherwise. 
\end{defi}

If the alphabet is
finite, for instance, the set $A^{-\mathbb{N}}$ is compact and
therefore, our uniform continuity is equivalent to asking that every
point is continuous by the Heine-Cantor Theorem. But this is not the
case in general since we are not assuming finiteness of the alphabet.

We now introduce the objects of interest of the present paper, which
are the stationary chains compatible with a given kernel $P$. 

\begin{defi}\label{def:compa} A stationary stochastic chain ${\bf
    X}=(X_{n})_{n\in\Z}$ of law $\mu$ on $A^{\mathbb{Z}}$ is said to
  be \emph{compatible} with a family of transition probabilities $P$
  if the later is a regular version of the conditional probabilities
  of the former, that is
\begin{equation}\label{compa}
\mu(X_{0}=a|X_{-\infty}^{-1}=a_{-\infty}^{-1})=P(a|a_{-\infty}^{-1})
\end{equation}
for every $a\in A$ and $\mu$-a.e. $a_{-\infty}^{-1}$ in $A^{-\mathbb{N}}$.
\end{defi}  

Standard questions, when we consider non-Markovian kernels, are
\begin{description}
\item [Q1.] Does there  exist a stationary chain compatible with $P$? 
\item [Q2.] Is this chain unique?
\end{description}
These questions can be answered using the powerful constructing method of
``perfect simulation \emph{via} coupling from the past''.
 For a stationary stochastic chain, an algorithm of perfect simulation
 aims to construct finite samples distributed according to the
 stationary measure of the chain.  \cite{propp/wilson/1996} introduced
 (in the Markovian case, that is, in the case where $P$ is a
 transition matrix) the coupling from the past (CFTP) algorithm. 
 This class of algorithms uses a sequence of i.i.d. random
 variable ${\bf U} = \{U_i\}_{i \in \Z}$, uniformly distributed in
 $[0,1[$, to construct a sample of the stationary chain.

 From now on, every chain will be constructed as a function of ${\bf U}$ and
 therefore, the only  probability space  $(\Omega,\mathcal{F},\mathbb{P})$ used
 in this paper is the one associated to this sequence of i.i.d. r.v.'s.\\
 
A CFTP algorithm is completely determined by the
 \emph{update function} $F$, with which are constructed a \emph{coalescence time} $\theta$ and a
 \emph{reconstruction function} $\Phi$. 

The update function  $F:A^{-\mathbb{N}}\cup A^{\star}\times[0,1[\rightarrow A$ has the property that for any $\underline{a}\in A^{-\mathbb{N}}$ and for any $a\in A$, $\mathbb{P}(F(\underline{a},U_{0})=a)=P(a|\underline{a})$.
Define the iterations of $F$ by
\begin{equation}\label{eq:F}
F_{[k,l]}(\underline{a},U_{k}^{l})=F\left(\underline{a}F_{[k,k]}(\underline{a},U_{k})F_{[k,k+1]}(\underline{a},U_{k}^{k+1})\ldots F_{[k,l-1]}(\underline{a},U_{k}^{l-1}),U_{l}\right),
\end{equation}
for any $-\infty<k\leq l\leq+\infty$,  where $F_{[k,k]}(\underline{a},U_{k})=F(\underline{a},U_{k})$. Based on these iterations, we define, for any window $\{m,\ldots,n\}$, $-\infty<m\leq n\leq+\infty$, its coalescence time as
\begin{equation}\label{eq:theta}
\theta[m,n]:=\max\{j\leq m:F_{[j,n]}(\underline{a},U_{j}^{n})\,\,\,\textrm{does not depend on }\,\underline{a}\},
\end{equation}
with $\theta[n]:=\theta[n,n]$. Finally, the reconstruction function of time $i$ is defined by
\begin{equation}\label{eq:Phi}
[\Phi({\bf U})]_{i}=F_{[\theta[i],i]}(\underline{a},U_{\theta[i]}^{i}).
\end{equation}

Given a kernel $P$, if we can find an $F$ such that $\theta[m,n]$ is
a.s. finite for any $-\infty<m\leq n<+\infty$, then, the reconstructed
sample $[\Phi({\bf U})]_{i}$, $i=\theta[m,n],\ldots,n$ is distributed
according to the unique stationary measure. A well-known consequence
of this constructive argument is that there exists a unique stationary
chain compatible with $P$, answering questions Q1 and Q2 at the
same time (see Corollary \ref{coro4} below). Observe that the choice
of the function $F$ is crucial in this approach, a ``bad'' choice
could lead to a coalescence time which is not a.s. finite or having
heavy tail distribution with no finite expectation.  But this choice
depends on the kernel, and in
particular, according to the assumptions made on the kernel, we might guarantee that
there exists a $F$ for which $\theta[m,n]$ is a.s. finite. Another
important observation is that, \emph{a priori}, such algorithms are
not practical in the sense that, at each steps, it requires that we
generate \emph{all} the pasts $\underline{a}$. For this reason, a
particular update function based on a \emph{length function} (see
\eqref{eq:Lfunction}) will be defined, allowing to decide, for some
pasts $\underline{a}$ and some values of $U_{0}$, what is the value of
$F(\underline{a},U_{0})$ looking only at a finite portion of
$\underline{a}$.

\section{Local continuity and good skeleton} \label{sec:assumptions}

\begin{defi}\label{def:contexttrees}

\begin{itemize}
\item A \emph{context tree} on a given alphabet $A$ is a subset of $A^{-\mathbb{N}}\cup A^{\star}$ which forms a partition of $A^{-\mathbb{N}}$ and for which if $v\in\tau$, then $uv\notin\tau$ for any $u\in A^{-\mathbb{N}}\cup A^{\star}$. 
\item
For any context tree $\tau$, we denote by $\leftexp{<\infty}{\tau}$ the set of contexts of $\tau$ having finite lengths and by $\leftexp{\infty}{\tau}$ the remaining contexts. Clearly, these subsets form a partition of $\tau$.
\item For any past $\underline{a}\in A^{-\mathbb{N}}$ we denote by $c_{\tau}(\underline{a})$ the unique element of $\tau$ which is suffix of $\underline{a}$.
\end{itemize}
\end{defi}

For our purposes, a particular class of context trees on $A$ will be of interest.

\begin{defi}
A context tree is a \emph{skeleton context tree} (or simply \emph{skeleton)} if it is the smallest context tree containing the set of its infinite length contexts. We also consider $\emptyset$ to be a skeleton.
\end{defi}
In order to illustrate this notion, let us give one simple example, on
$A=\{-1,+1\}$. Let
\begin{equation}\label{arvore}
\underline{\tau}:=\{1,1(-1),1(-1)(-1),1(-1)(-1)(-1),\ldots\}\cup\{-\underline{1}\}.
\end{equation}
and
\begin{eqnarray}
  \tau_g& :=& \{11,(-1)1,11(-1),(-1)1(-1),11(-1)(-1),(-1)1(-1)(-1), \label{arvore1} \\
 && \quad
 11(-1)(-1)(-1),(-1)1(-1)(-1)(-1)\ldots\}\cup\{-\underline{1}\}. \nonumber
\end{eqnarray}
Observe that $\underline{\tau}$ and $\tau_g$ are indeed context
trees (i.e. satisfy the requirements of the first item of Definition
\ref{def:contexttrees}). However, the only infinite length context in
both trees is $-\underline{1}$, and it is easy to see that
$\underline{\tau}$ is the smallest context tree having this unique
infinite length context. Therefore, $\underline{\tau}$ is a skeleton
whereas $\tau_g$ is not. We will come back several times to this
skeleton along the paper.

The reason why we introduced skeletons is that they give us a nice way to formalize the notion of localized continuity, which extend the continuity
assumption. In Section \ref{sec:examples1} 
this notion is explained by  mean of examples.

\begin{defi}\label{def:LC}
  A kernel $P$ belongs to the class of \emph{locally continuous
    kernels with respect to the skeleton $\tau$} if for
  any $v\in\leftexp{<\infty}{\tau}$, the sequence $\{\alpha_{k}^{v}\}_{k\ge0}$ defined by
\begin{equation}\label{condition1}
\alpha_{k}^{v}:=\inf_{a_{-k}^{-1}\in A^{k}}\sum_{a\in A}\inf_{\underline{z}}P(a|v\,a_{-k}^{-1}\,\underline{z})\,,\,\,k\ge0
\end{equation}
converges to $1$.
We will denote this class by   \emph{LC($\tau$)}. The \emph{probabilistic skeleton (p.s.) of $P$}  is the pair $(\tau,p)$ where $p:=\{p(a|v)\}_{a\in A,\,v\in\leftexp{<\infty}{\tau}}$,
\begin{equation}\label{eq:p}
p(a|v):=\inf_{\underline{z}}P(a|v\underline{z})
\end{equation}
and $p(a|v)=P(a|v)$ for any $v\in \leftexp{\infty}{\tau}$.
\end{defi}

Some observations on the above definitions.

\begin{obs}
\emph{
\begin{enumerate}
\item If $P$ is LC($\tau$) then,  all pasts $\underline{a}$ such that $|c_{\tau}(a)|<+\infty$ are continuity point for $P$. On the other hand, we 
require nothing on the points $\underline{a}$ such that $|c_{\tau}(\underline{a})|=\infty$.  In practice, we will see later that these will be the pasts with slow continuity rate (or even the discontinuous pasts).
\item Observe that for any  fixed $v\in\leftexp{<\infty}{\tau}$,
  $\{p(a|v)\}_{a\in A}$ needs not to be a probability distribution on $A$. 
\end{enumerate}
}
\end{obs}

Our first main assumption for our results will be that $P$ is a  probability kernel on $A=\{1,2,\ldots\}$ being locally continuous with some p.s. $(\tau,p)$. We will furthermore require that this p.s. is ``good'' in a sense we explain now.  We first introduce 
sequences of random variables which are obtained as coordinatewise  functions of the sequence ${\bf U}$. The first sequence, ${\bf Y}$ is defined as follows, for any $i\in\mathbb{Z}$
\begin{equation}
Y_{i}=\sum_{a\in A}a.{\bf 1}\left\{\sum_{j=0}^{a-1}\alpha(j)\leq U_{i}<\sum_{j=0}^{a}\alpha(j)\right\}+\star.{\bf 1}\{U_{i}\ge\alpha_{-1}\}
\end{equation}
where $\alpha(0):=0$.
 For any ${\bf
  a}=a_{-\infty}^{+\infty}\in A^{\mathbb{Z}}$, we also define the
sequence of r.v.'s 
${\bf Y}({\bf a})$ where for any $i\in\mathbb{Z}$,
$$Y_{i}({\bf a}):=Y_{i}.{\bf 1}\{Y_{i}\in A\}+a_{i}.{\bf 1}\{Y_{i}=\star\}.$$
Finally, we define the sequence $\{c_{\tau}^{n}\}_{n\in\mathbb{Z}}$ of maximum context
length (based on ${\bf Y}$) as
\begin{equation}\label{eq:ctau}
c_{\tau}^{n}:=\sup_{{\bf a}}|c_{\tau}(Y_{-\infty}^{n}({\bf a}))|\,,\,\,\,n\in\mathbb{Z},
\end{equation}
where the notation $c_{\tau}(\cdot)$ was introduced in Definition \ref{def:contexttrees}. Observe that the event $\{c_{\tau}^{n}\leq k\}$ is
$\mathcal{F}(Y_{n-k+1}^{n})$-measurable.  
\begin{defi}\label{def:good}
Any time belonging to the set
\begin{equation}\label{eq:good1}
\{i\leq m:Y_{j}\in A\,\,\,\textrm{or}\,\,\,\,c_{\tau}^{j}\le j-i\,,\,\,\,j=i,\ldots,m\}
\end{equation}
is called a \emph{good coalescence time for time $m$}. We say that $(\tau,p)$ is a \emph{good probability skeleton} if $\bar{\theta}[0]$, the supremum over the set \eqref{eq:good1} when $m=0$, has finite expectation. 
\end{defi}

\begin{obs}
\emph{
\begin{enumerate}
\item Any good coalescence time is measurable with
respect to $\mathcal{F}(Y_{-\infty}^{m})$ (less information than
$U_{-\infty}^{m}$).
\item Assuming that $(\tau,p)$ is a good p.s. of
$P$ implies that there exists a set $A(\tau) \subset A$ such that  
$\inf_{\underline{z}}P(a|\underline{z})>0$ for all $a \in
A(\tau)$. In other words, this means that we assume \emph{weak non-nullness} for $P$. 
\item Observe that, under the assumption that $\tau$ is a good p.s., we have $c^{j}_{\tau}<+\infty, \mathbb{P}$-a.s. for any $j\in\mathbb{Z}$.
\end{enumerate}}
\end{obs}

\section{Main result and direct consequences} \label{sec:statements}

We will say that a non-negative sequence $\{c_{n}\}_{n\in\mathbb{N}}$
decays exponentially fast to zero if there exist a constant $D>0$ and
a real number $0<d<1$ such that $c_{n}\leq Dd^{n}$ for any $n$. We say
that a real-valued random variable $W$ has exponential tail if
$\{\mathbb{P}(|W|\geq n)\}_{n\in\mathbb{N}}$ decays exponentially fast
to zero, and summable tail if $\{\mathbb{P}(|W|\geq
n)\}_{n\in\mathbb{N}}$ is summable.

For any skeleton $\tau$, let 
 \[
 N(\tau):=\{i\ge1:\exists v\in\tau, \,|v|=i\}.
 \]
 In the sequel, one of the main characteristics of a kernel in
 LC($\tau$) will be the sequence of sequences
 $\{\{\alpha_{k}^{i}\}_{k\ge-1}\}_{i\in N(\tau)}$, defined as  
\begin{equation}\label{eq:alphaki}
\alpha_{k}^{i}:=\inf_{v\in\leftexp{\leq i}{\tau}}\alpha_{k}^{v},
\end{equation}
where $\leftexp{\leq i}{\tau}$ denotes the set of contexts in $\tau$ having length smaller or equal to $i$. Observe that there is a notational similarity between the case where  the exponent  is an
integer $i$ and the case where it is  an element $v$ of $\tau$.

\begin{theo}\label{theo2}
  Consider a kernel $P$ belonging to LC($\tau$) and assume that its
  probabilistic skeleton $(\tau,p)$ is good with good coalescence
  times $\bar{\theta}[0]$ for time $0$. Let $A_{0}:=\alpha_{-1}$ and for any $k\ge1$, denote
$$A_{k}:=\left\{1-(\mathbb{E}|\bar{\theta}[0]|+1)\mathbb{P}(U_{0}>\alpha_{k}^{c_{\tau}^{-1}})\right\}\vee \alpha_{-1}.$$
Then, we can construct for $P$, an update
  function $F$ and a corresponding coalescence time $\theta$ such
  that 
\begin{enumerate}[(i)]
\item  If $\sum_{k\geq1}\prod_{j=0}^{k-1}A_{k}=+\infty$, then $\theta[0]$ is $\mathbb{P}$-a.s. finite.
\item If  $\sum_{k\ge0}(1-A_{k})<+\infty$, then $\theta[0]$ has summable tail.
\item If $\bar{\theta}[0]$ has exponential tail and $\{1-A_{k}\}_{k\geq0}$ decays exponentially fast to zero, then $\theta[0]$ has exponential tail.
\end{enumerate}
In particular, in each of these regimes, the CFTP with update function $F$ is feasible.
\end{theo} 

The proof of this result is given in Section \ref{sec:prova2}. Section \ref{sec:examples1} will discuss  explicit examples. We now
state some direct consequences of Theorem \ref{theo2}.

Theorem \ref{theo2} states, in particular, that the CFTP algorithm is feasible with some function $F$ (which will be constructed in the proof). We recall that this means that the algorithm constructs, for any $i\in\mathbb{Z}$, an almost surely finite sample $[\Phi({\bf U})]_{\theta[i]}^{i}$, which is a deterministic function of ${\bf U}$. In the sequel, we will often write $X_{i}$ for $[\Phi({\bf U})]_{i}$ (and ${\bf X}$ for $[\Phi({\bf U})]_{-\infty}^{+\infty}$) in order to avoid overloaded notations, keeping in mind the fact that for any $i$, $X_{i}$ is constructed as a deterministic function of ${\bf U}$. Actually, by Theorem \ref{theo2},  $X_{i}$ depends only on a $\mathbb{P}$-a.s. finite part of this sequence since  $X_{i}:=[X(\ldots,u_{\theta[i]-1},U_{\theta[i]},\ldots,U_{i},u_{i+1},\ldots)]_{i}$ for any ${\bf u}\in [0,1[^{\mathbb{Z}}$. 
As we already mentioned in Section \ref{sec:notation}, the existence of a perfect simulation algorithm has important consequences, stated in the two next corollaries, and whose proofs use standard arguments given, for example, in \cite{comets/fernandez/ferrari/2002}. 

\begin{coro}\label{coro4}\emph{(Existence and uniqueness)}.
${\bf X}$ is stationary and compatible with $P$. Denoting by $\mu$ the
stationary measure of ${\bf X}$ we have
\[
\mu(\cdot):=\mathbb{P}(\Phi({\bf U})\in \cdot).
\]
Moreover, $\mu$ has support on the set of continuous pasts of $P$.
\end{coro}

\vspace{0.3cm}

When $\mathbb{E}|\theta[0]|<+\infty$, 
the chain ${\bf X}$ exhibits a \emph{regeneration scheme}.
We call time $t$ a regeneration  time for the chain ${\bf X}$ if  $\theta[t,+\infty]=t$.
Define the chain $\boldsymbol{\mathcal{T}}$ on $\{0,1\}$ by $\mathcal{T}_{j}:={\bf 1}\G\{j=\theta[j,+\infty]\D\}$.
Then, consider the sequence of time indexes ${\bf T}$ defined by $\mathcal{T}_{j}=1$ if and only if $j=T_{l}$ for some $l$ in $\mathbb{Z}$, $T_{l}< T_{l+1}$ and with the convention $T_{0}\leq0<T_{1}$. We say that ${\bf X}$ has a regeneration scheme if $\boldsymbol{\mathcal{T}}$ is a renewal chain (that is, if the  increments $(T_{i+1}-T_{i})_{i\in\mathbb{Z}}$ are independent, and are identically distributed for $i\neq0$).

\begin{coro}\label{coro5}\emph{(Regeneration scheme)}.
  Under conditions (ii) and (iii) of Theorem \ref{theo2},
  the chain ${\bf X}$ has a regeneration scheme. The random strings
  \newline$(\Phi({\bf U})_{T_{i}},\ldots,\Phi({\bf U})_{T_{i+1}-1})_{i\neq0}$
  are i.i.d. and have finite expected size. Under the stronger
  requirement of (iii), the lengths of these strings have exponential
  tail.
\end{coro}

In words, this corollary states that the unique stationary chain
compatible with $(\tau,p)$ under the conditions of Theorem \ref{theo2}
can be viewed as an i.i.d. concatenation of strings of symbols of $A$
having finite expected size. A similar result has been first obtained
in \cite{lalley/1986} for one dimensional Gibbs states under
appropriate conditions on the continuity rate, and then in
\cite{comets/fernandez/ferrari/2002} under weaker conditions. Our
result strengthen all these results, and in particular, since
continuity is not assumed here, our chains are not even necessarily
Gibbsian. This lack of Gibbsianity is easy to establish, using the
recent result of \cite{fernandez/gallo/maillard/2011}, in which it has
been shown that the unique stationary chain compatible with the
p.s. $(\underline{\tau},p)$ ($\underline{\tau}$ is
the tree corresponding to a regeneration process) is not always
Gibbsian, even when it satisfies continuity and $\alpha(a)>0$ for all $a\in A$.

\section{Applications}\label{sec:examples1}

In Section \ref{sec:howreads} we will  explain how
Theorem \ref{theo2} reads in two special cases of local continuity which are the strong local continuity and the uniform local continuity. 
Then, we will present several
examples that  illustrate our results: in Section \ref{sec:specificskeletons} we consider local continuity with respect to two special cases of skeletons and finally, Section \ref{sec:ex} is dedicated to explicit examples.

\subsection{Specific local continuities} \label{sec:howreads}

\paragraph{Uniform local continuity.}

\begin{defi}
 A kernel $P$ belongs to the class of \emph{uniformly local continuous
    kernels with respect to the skeleton $\tau$} if
  \begin{equation}\label{eq:ULC}
\alpha_{i}:=\inf_{v\in\leftexp{<\infty}{\tau}}\alpha_{i}^{v}\stackrel{i\rightarrow+\infty}{\longrightarrow}1.
\end{equation}
We will denote this class by   \emph{ULC($\tau$)}.
\end{defi}

When $P$ belongs to ULC($\tau$), we can use the facts that $\alpha_{i}:=\inf_{k\in N(\tau)}\alpha^{k}_{i}$ converges to $1$, and that $U_{0}$ is independent of $c^{-1}_{\tau}$ (since this later is $\mathcal{F}(U_{-\infty}^{-1})$-measurable) to obtain
\begin{align}
\mathbb{P}(U_{0}\ge \alpha_{i}^{c_{\tau}^{-1}})
&=\sum_{k\in N(\tau)}\mathbb{P}(U_{0}\ge \alpha_{i}^{k})\mathbb{P}(c_{\tau}^{-1}=k)\leq\mathbb{P}(U_{0}\ge \alpha_{i})\sum_{k\in N(\tau)}\mathbb{P}(c_{\tau}^{-1}=k)\leq(1-\alpha_{i}).
\end{align}
We have therefore proved the following corollary.
\begin{coro}\label{coro:ulc}
Restricting the assumptions of Theorem \ref{theo2} to the case of ULC($\tau$) kernels, the same statements hold substituting $A_{m}$ by $1-\mathbb{E}(|\bar{\theta}[0]|+1)(1-\alpha_{m})$, for any $m\ge1$.
\end{coro}
As we will explain in Section \ref{sec:specificskeletons}, Theorem 4.1 in \cite{comets/fernandez/ferrari/2002} and Theorem 1 in \cite{gallo/garcia/2010} are particular cases of uniform local continuity where $\tau=\emptyset$ and $\tau$ has a terminal string (see Definition \ref{def:terminalstring}), respectively. 

Explicit examples of this regime are given by Examples \ref{ex:AR}, \ref{ex:ita} and \ref{ex:extAR}.

\paragraph{Strong local continuity.}

\begin{defi}\label{def:slc}
 A kernel $P$ belongs to the class of \emph{strongly local continuous
    kernels with respect to the skeleton $\tau$} if for any $v\in\leftexp{<\infty}{\tau}$, there exists a positive integer $h(v)$ such that
for any $k\geq h(v)$, $\alpha_{k}^{v}=1$.
We will denote this class by   \emph{SLC($\tau$)}. 
\end{defi}
These kernels belong to a particular
  class of probability kernels known in the literature under the name
  of \emph{probabilistic context trees}, which have been introduced by
  \cite{rissanen/1983}.
  
  When $P$ belongs to SLC($\tau$) on a finite alphabet, we can use the fact that for any $i\in N(\tau)$, $\alpha_{k}^{i}=1$ for any $k\ge h(i)$, where $h(i):=\sup_{v\in\leftexp{\leq i}{\tau}}h(v)$, and we obtain, using $h^{-1}(i):=\inf\{k\ge1:h(k)>i\}$,
  \begin{align}\label{eq:upper}
\mathbb{P}(U_{0}\ge \alpha_{i}^{c_{\tau}^{-1}})\leq\mathbb{P}(c_{\tau}^{-1}> h^{-1}(i))=\mathbb{P}(c_{\tau}^{0}> h^{-1}(i)).
\end{align}
We have therefore proved the following corollary.
\begin{coro}\label{coro:slc}
Restricting the assumptions of Theorem \ref{theo2} to the case of SLC($\tau$) kernels on finite alphabet, the same statements hold substituting $A_{m}$ by $1-\mathbb{E}(|\bar{\theta}[0]|+1)\mathbb{P}(c_{\tau}^{0}>h^{-1}(m))$, for any $m\ge1$.
\end{coro}
As we will explain in Section \ref{sec:specificskeletons}, the results
of \cite{gallo/2009} are a particular case of strong local continuity
where $\tau$ has a terminal string (see Definition
\ref{def:terminalstring}). Owing to Corollary \ref{coro4}, 
 the compatible stationary measure has support on the set of
finite length contexts.  An explicit example of such regime is given
by Example \ref{ex:eu}.

\subsection{Specific skeletons}\label{sec:specificskeletons}

{\bf Skeleton $\tau=\emptyset$}

\vspace{0.2cm}

\cite{comets/fernandez/ferrari/2002}  assumed that $\sum_{a\in A}\inf_{\underline{z}}P(a|\underline{z})>0$ (weak non-nullness), and that the sequence $\{\omega_{k}\}_{k\geq0}$, defined by (\ref{eq:continuity}), satisfies
\begin{equation}\label{regimefinite}
\sum_{k\geq1}\prod_{j=0}^{k-1}\omega_{k}=+\infty. 
\end{equation}
This implies, in particular, that $P$ is continuous. Without further information on $P$, this uniform convergence assumption gives us the possibility of using \emph{any} skeleton $\tau$.
In order to fix ideas, we will choose $\tau=\emptyset$, which is the
simplest skeleton and we have $P$ is ULC($\tau$). Thus we
have $\alpha_{k}:=\omega_{k}$ for any $k\ge0$ and $p:=\{p(a|v)\}_{a\in
  A,\,v\in\leftexp{<\infty}{\tau}}$ is in fact $\{\alpha(a)\}_{a\in
  A}$. 
In
this case, since all the contexts of $\tau=\emptyset$ have length $0$,
we have $\bar{\theta}[0]=0$,
$\mathbb{E}|\bar{\theta}[0]|+1=1$ and thus,
$A_{k}=\alpha_{k}=\omega_{k}$. This shows that our Corollary \ref{coro:ulc} retrieves the results
of \cite{comets/fernandez/ferrari/2002}. 

\vspace{0.3cm}

{\bf Skeletons with a terminal string}

\vspace{0.1cm}

\begin{defi}\label{def:terminalstring}
We say that $w$ is a terminal string for a skeleton $\tau$ if  for any $v\in\leftexp{<\infty}{\tau}$
we have $v_{-|v|+i}^{-|v|+|w|-1+i}\neq w,\,\,i=1,\ldots,
|v|-|w|$. 
\end{defi}

\begin{prop}\label{prop:skel1}
Consider a probabilistic skeleton $(\tau,p)$ for which $\tau$ has a
terminal context $w$, and $p=\{p(a|v)\}_{a\in
  A,\,v\in\leftexp{<\infty}{\tau}}$ satisfies
$\inf_{i=1,\ldots,|w|}\inf_{v\in\tau}p(w_{-i}|v)=\epsilon$ for some
$\epsilon>0$. These p.s.'s are good and the corresponding  good
coalescence time  has exponential tail.
In the particular case where $|w|=1$, we have $\mathbb{P}(\bar{\theta}[0]\leq-n)= (1-\epsilon)^{n}$, $n\ge1$.
\end{prop}
The proof of this proposition is immediate once we observe that  
\[
\bar\theta[0]\ge \sup\{i\leq -|w|+1: Y_{i}^{i+|w|-1}=w\}.
\]
We have the following for skeletons with a terminal string.
\begin{itemize}
\item When $P$ is in SLC($\tau$), Corollary  \ref{coro:slc} extends Theorem 5.1 of \cite{gallo/2009} (as explained by Example \ref{ex:eu} below).
\item When $P$ is in ULC($\tau$), Corollary \ref{coro:ulc} extends Theorem 1 of \cite{gallo/garcia/2010}.
\end{itemize}

\subsection{Examples}\label{sec:ex}

Several examples of continuous and discontinuous kernels can be found
in the literature of perfect simulation for chains with infinite
memory (we refer to \cite{comets/fernandez/ferrari/2002},
\cite{gallo/2009} or \cite{desantis/piccioni/2012} for instance). 

Kernels that are locally continuous are not necessarily
  discontinuous neither necessarily continuous. An important aspect of
  the present work is that we not only want to consider
  discontinuous kernels, but also to ``speed up'' the CFTP algorithms
  that have been proposed in the literature in the sense that the tail
  distribution of the coalescence time $\theta[0]$ of our CFTP will
  decay faster to zero. In some cases, it will decay to zero (and
  therefore $\theta[0]$ will be a.s. finite) while  other CFTP's of the literature will not be feasible since
  their  coalescence 
  time are not a.s. finite.

We now present five examples on the binary alphabet $A=\{-1,+1\}$. 

\subsubsection{Using $\emptyset$ as skeleton}\label{sec:CFFexample}
\begin{example}\label{ex:AR} 
\emph{Our first example is   the
well-known binary auto-regressive processes (AR), which is precisely the main example 
presented in \cite{comets/fernandez/ferrari/2002}. 
 These models are defined
using a continuously differentiable increasing function
$\psi:\mathbb{R}\rightarrow]0,1[$ and an absolutely summable sequence of real
numbers $\{\theta_{n}\}_{n\geq0}$:
\[
P(1|\underline{a}):=\psi\left(\theta_{0}+\sum_{k\geq1}\theta_{k}a_{-k}\right)\,,\,\,\,\,\forall \underline{a}\in A^{-\mathbb{Z}}. 
\]
Such kernels are continuous since for any $\underline{a}$, we have
\begin{align*}
\omega_{k}(\underline{a})&=\inf_{\underline{z}}P(1|a_{-k}^{-1}\underline{z})+\inf_{\underline{z}}P(-1|a_{-k}^{-1}\underline{z})\\
&=\inf_{\underline{z}}\psi\left(\theta_{0}+\sum_{i=1}^{k}a_{-i}\theta_{i}+\sum_{i\ge k+1}z_{-i+k}\theta_{i}\right)+1-\sup_{\underline{z}}\psi\left(\theta_{0}+\sum_{i=1}^{k}a_{-i}\theta_{i}+\sum_{i\ge k+1}z_{-i+k}\theta_{i}\right)\\
&=\psi\left(\theta_{0}+\sum_{i=1}^{k}a_{-i}\theta_{i}-\sum_{i\ge k+1}|\theta_{i}|\right)+1-\psi\left(\theta_{0}+\sum_{i=1}^{k}a_{-i}\theta_{i}+\sum_{i\ge k+1}|\theta_{i}|\right)
\end{align*}
and using the mean valued theorem, we have 
\[
\omega_{k}(\underline{a})=1-2\psi'(c)\sum_{i\ge k+1}|\theta_{i}|
\] 
for some real number $c=c(a_{-k}^{-1})$ in the interval
\[
\left[\theta_{0}+\sum_{i=1}^{k}a_{-i}\theta_{i}-\sum_{i\ge k+1}|\theta_{i}|\,\,,\,\,\,\,\theta_{0}+\sum_{i=1}^{k}a_{-i}\theta_{i}+\sum_{i\ge k+1}|\theta_{i}|\right].
\]
Due to the assumption that $\{\theta_{n}\}_{n\geq0}$ is an absolutely summable sequence of real
numbers, we have that $P$ is continuous in every $\underline{a}$. Moreover, we observe that the rate at which $\omega_{k}(\underline{a})$ converges to $1$ is controlled (exclusively) by the rate at which $r_{k}:=\sum_{i\ge k+1}|\theta_{i}|$ converges to $0$, independently of $\underline{a}$. This is because for any $\underline{a}$, $\psi'(c(a_{-k}^{-1}))$ converges to a positive constant (if we exclude the trivial case of $\psi$ constant). In other words, denoting by $c:=\sup_{\underline{a}}\psi'(c)$ and $C:=\inf_{\underline{a}}\psi'(c)$, we have
\begin{equation}\label{eq:1111}
1-2c\sum_{i\ge k+1}|\theta_{i}|\leq \omega_{k}(\underline{a})\leq 1-2C\sum_{i\ge k+1}|\theta_{i}|
\end{equation}
showing that there is nothing to gain in using the notion of local continuity. Our conditions for perfect simulation of this model are, using Corollary \ref{coro:ulc} to the case where $\tau=\emptyset$, the same as in \cite{comets/fernandez/ferrari/2002}, as we said in the first part of Section \ref{sec:specificskeletons}.}
\end{example}

\subsubsection{A continuous case for which $\tau=\emptyset$ is not enough}

In certain conditions, $\omega_{k}$ may increase very slowly to $1$. In these cases, the algorithm of \cite{comets/fernandez/ferrari/2002} can be very slow to stop, or may not be feasible. That is, the coalescence time $\theta[0]$, or, roughly speaking, the random number of steps operated by the algorithm, may have no first moment or even may not be almost surely finite.  For a given kernel $P$, a simple reason for which $\omega_{k}$ could increase slowly to $1$ is that some pasts $\underline{a}$ could have a very slow continuity rate. Since the definition of $\omega_{k}$ is uniform on the pasts, it has to take into account these \emph{bad} pasts as well. 
Providing we have the information of the position of these \emph{bad} pasts, our method allows us to create a skeleton $\tau$ for which the set of infinite size contexts is composed by these bad pasts, and to work separately on the problem of the resulting p.s. asking whether it is good or not (according to Definition \ref{def:good}). We observed that in the  example of the AR processes, every past $\underline{a}$ has the same continuity rate, and therefore, the only natural skeleton was $\emptyset$. We now present an example in which this is not the case.

\begin{example}\label{ex:ita}
\emph{This example is an unpublished example presented in \cite{desantis/piccioni/2010}. First, define for any $\sigma\in(0,1)$ and any $\underline{a}\in \{-1,1\}^{-\mathbb{N}}$
\[
T_{\sigma}(\underline{a}):=\inf\{k\geq1:\frac{1}{k}\sum_{i=1}^{k}{\bf1}\{a_{-i}=1\}\geq\sigma\},
\]
with the convention that $T_{\sigma}(\underline{a})=+\infty$ if the set of indexes is empty.
This is the first time the proportion of $1$'s is larger than $\sigma$, when we look backwards in the sequence $\underline{a}$. 
Then, consider two summable sequences  $\{\beta(i)\}_{i\geq1}$ and $\{\gamma(i)\}_{i\geq1}$, such that $\gamma(i)\leq\beta(i)$ and three real numbers $b_{1}\in (0,1)$, $c>0$ and $\sigma>0$ satisfying
\[
b_{1}.\left(1-c\sum_{i\geq1}\beta(i)\right)>\sigma.
\]
The kernel $P$ on $\{-1,1\}$ is defined by
\begin{equation}\label{eq:kernel}
P(1|\underline{a})=b_{1}\left(1-c\sum_{i\geq1}\Big(\beta(i){\bf 1}\{a_{-i}=-1,\,\,T_{\sigma}(\underline{a})>i\}+\gamma(i){\bf 1}\{a_{-i}=-1,\,\,T_{\sigma}(\underline{a})\leq i\}\Big)\right).
\end{equation}
We observe that for pasts $\underline{a}$ such that $T_{\sigma}(\underline{a})=\infty$, the continuity rate is controlled by $\{\beta(i)\}_{i\geq1}$, since $\omega_{k}(\underline{a})=1-b_{1}c\sum_{i\ge k+1}\beta(i)$, while for pasts $\underline{a}$ such that $T_{\sigma}(\underline{a})<\infty$, the continuity rate is controlled by $\{\gamma(i)\}_{i\geq1}$, since $\omega_{k}(\underline{a})=1-b_{1}c\sum_{i\ge k+1}\gamma(i)$ for $k\ge T_{\sigma}(\underline{a})$.  Contrary to Example \ref{ex:AR}, we may have something to gain in using the notion of local continuity, depending on the tails of the series of both sequences. As we will see, a natural choice for the skeleton is
\[
\tau^{\sigma}:=\bigcup_{\underline{a}}a_{-T_{\sigma}(\underline{a})}^{-1}.
\]
\begin{prop}\label{prop:skel2}
For any $\sigma\in(0,1)$, the p.s. $(\tau^{\sigma},p)$, where $\inf_{v\in\tau^{\sigma}}p(1|v)>\sigma$, is good, and has a good coalescence time $\bar{\theta}[0]$ with exponential tail.
\end{prop}
\begin{proof}
The chain ${\bf Y}$ takes value $1$ with probability $\sigma$ and $\star$ with probability $1-\sigma$. By the definition of $\tau^{\sigma}$, we see that $\bar{\theta}[0]:=\min\{i\geq1:\frac{1}{i}\sum_{j=-i}^{-1}Y({\bf 0})_{j}\ge\sigma\}$ is a stopping time in the past of the form of \eqref{eq:good1}. 
On the other hand, we observe that this random variable has the same tail distribution as $\underline{\theta}[0]:=\min\{i\geq1:\frac{1}{i}\sum_{j=0}^{i-1}Y({\bf 0})_{j}\ge\sigma\}$ since ${\bf Y}({\bf 0})$ is i.i.d. Then, 
\[
\mathbb{P}(\bar{\theta}[0]<-N)=\mathbb{P}(\underline{\theta}[0]>N)\leq \mathbb{P}\left(\frac{1}{N}\sum_{j=0}^{N-1}Y({\bf 0})_{j}<\sigma\right)
\]
which decays exponentially by the well-known \emph{Chernoff bound}.
\end{proof}
}

\emph{
Now, we notice that, for any $i\in N(\tau)$
\[
\alpha_{k}^{i}=\inf_{a_{-k-i}^{-1}:T(\underline{a})\leq k+i}\sum_{a\in A}\inf_{\underline{z}}P(a|a_{-k-i}^{-1}\,\underline{z})=1-b_{1}c\sum_{j\geq k+i}\gamma(j),
\]
which in turns implies that $\alpha_{k}\geq\inf_{i\geq1}[1-b_{1}c\sum_{j\geq k+i}\gamma(j)]\geq1-b_{1}c\sum_{j\geq k}\gamma(j)$.
It follows from the summability of $\{\gamma(i)\}_{i\geq1}$ that $P\in ULC(\tau^{\sigma})$. 
Therefore,   using Proposition \ref{prop:skel2}, we can apply Corollary \ref{coro:ulc} to show that perfect simulation can be done without assuming $\sum_{i\geq1}i\gamma(i)<+\infty$ (assumption required by \cite{desantis/piccioni/2010}).
Moreover, depending on the rate at which $b_{1}c\sum_{j\geq k+1}\gamma(j)$ converges to $0$, we obtain several regimes stated in Theorem \ref{theo2}. All this occurs independently of the sequence $\{\beta(i)\}_{i\geq1}$, which is the sequence that controls the tail of the CFTP in \cite{comets/fernandez/ferrari/2002}.
}
\end{example}

\subsubsection{Three examples in $LC(\underline{\tau})$}\label{sec:discontinuousexamples}
Let $\mathcal{L}(\underline{a})$ denotes the first time we see a $1$ when we
look backward in $\underline{a}$. Formally, $\mathcal{L}(-\underline{1}):=+\infty$ and for any $\underline{a}\neq -\underline{1}$,
\[
\mathcal{L}(\underline{a}):=\inf\{i\ge1:a_{-i}=1\}.
\]

\begin{example} \emph{
The following example was presented in \cite{desantis/piccioni/2012} (see Example 2 therein). It is a kernel  belonging to $LC({\underline{\tau}})$, where $\underline{\tau}$ is defined by \eqref{arvore}, but does not belong to $ULC(\underline{\tau})$ nor $SLC(\underline{\tau})$.
Let $P(1|-\underline{1})=\epsilon>0$, and for any $\underline{a}\neq -\underline{1}$
\[
P(a|\underline{a})=\epsilon+(1-2\epsilon)\sum_{n\ge1}{\bf 1}\{a=a_{-\mathcal{L}(\underline{a})-n}\}q_{n}^{\mathcal{L}(\underline{a})},
\]
where, for any $l\ge1$, $\{q_{n}^{l}\}_{n\ge1}$ is a probability distribution on the integers.
 This kernel  has a discontinuity along $-\underline{1}$, in fact, we have
\begin{align*}
\omega_{k}=\omega_{k}(-\underline{1})&=\inf_{\underline{z}}P(1|-1_{-k}^{-1}\underline{z})+1-\sup_{\underline{z}}P(1|-1_{-k}^{-1}\underline{z})\\
&=\epsilon+1-\epsilon-(1-2\epsilon)=2\epsilon<1.
\end{align*}
On the other hand, it belongs to $LC(\underline{\tau})$ since for any $j\ge1$
\begin{align*}
\alpha_{k}^{j}&:=\inf_{v\in\leftexp{\leq j}{\underline{\tau}}}\inf_{a_{-k}^{-1}\in A^{k}}\sum_{a\in A}\inf_{\underline{z}}P(a|va_{-k}^{-1}\underline{z})\\
&=\inf_{l\leq j}\inf_{a_{-k}^{-1}\in A^{k}}\sum_{a\in A}\inf_{\underline{z}}P(a|(-1)^{l-1}\,1\,a_{-k}^{-1}\underline{z})=2\epsilon+(1-2\epsilon)\inf_{l\leq j}\sum_{i=1}^{k}q_{i}^{l}
\end{align*}
which goes to $1$ since for any $j\ge1$, $\{q_{i}^{j}\}_{i\ge1}$ is a probability distribution. Since the p.s. is $(\underline{\tau},p)$ with $p(a|v)\ge\epsilon$ for any $v\in\underline{\tau}$, and since $\underline{\tau}$ has $1$ as terminal string, it follows by Proposition \ref{prop:skel1} that it is a good p.s., with $\bar{\theta}[0]$ satisfying $\mathbb{P}(\bar{\theta}[0]\leq-n)=(1-\epsilon)^{n}$.  By Theorem \ref{theo2}, the tail distribution of the CFTP is related to the the tail distribution of $1-A_{k}:=(\mathbb{E}|\bar{\theta}[0]|+1)\mathbb{P}(U_{0}>\alpha_{k}^{c_{\underline{\tau}}^{-1}})
$,
and using the expression obtained above for $\alpha_{k}^{j}$, we obtain
\[
1-A_{k}=\frac{(1-2\epsilon)(1+\epsilon)}{\epsilon}[1-\sum_{j\ge1}\mathbb{P}(c_{\underline{\tau}}^{-1}=j)\inf_{l\leq j}\sum_{i=1}^{k}q_{i}^{l}].
\]
  In order to fix ideas, we will take, as suggested by \cite{desantis/piccioni/2012}, $\sum_{i=1}^{k}q_{i}^{l}\sim 1-(1-j^{-a})^{k+1}$, $l,k\ge1$,  with $a>0$. We then have $\inf_{l\leq j}\sum_{i=1}^{k}q_{i}^{l}\sim 1-(1-j^{-a})^{k+1}$ and we observe that in this case, $\inf_{i\ge1}\alpha_{k}^{i}=2\epsilon$ for any $k\ge1$, meaning that this kernel does not belong to ULC($\underline{\tau}$), and for any $i$, $\alpha_{k}^{i}<1$ for all $k\ge1$, meaning that it does not belong to SLC($\underline{\tau}$) neither. In order to prove that a CFTP is feasible, \cite{desantis/piccioni/2012} use the assumptions that $a<1$ and that $a_{\infty}+\epsilon>1$.}
\emph{In our case, using the fact that $\mathbb{P}(c_{\underline{\tau}}^{-1}=j)=(1-\epsilon)^{j-1}\epsilon$, we obtain
 \[
 1-A_{k}\leq\sum_{j\ge1}(1-\epsilon)^{j-1}(1-j^{-a})^{k+1},
 \]
and it follows that this kernel is in the regime (ii) of Theorem \ref{theo2}, since}
\begin{align*}
\sum_{k\ge0}(1-A_{k})&\leq\sum_{j\ge1}(1-\epsilon)^{j-1}j^{a}<\infty\,,\,\,\,\forall a>0.
\end{align*}
\emph{Observe that the restrictions that $a<1$ and that $a_{\infty}+\epsilon>1$ do not appear here.}
\end{example}

\begin{example}\label{ex:extAR}
\emph{
We now propose a simple extension of the AR processes which allows to choose
different models according to the past we consider.  It provides us with an example of kernel belonging to $ULC(\underline{\tau})$. 
Assume that we have two models
parametrized by $(\psi, \{\theta_{n}\}_{n\ge0})$ and $(\bar{\psi},
\{\bar{\theta}_{n}\}_{n\ge0})$, the ``standard model'' and the
``alternative model'' respectively.  Now, suppose that we choose the
standard model when $\mathcal{L}(\underline{a})$ is odd and the
alternative one otherwise. That is,
\[
P(1|\underline{a}):= \left\{ \begin{array}{ll}
\psi\left(\theta_{0}+\sum_{k\geq1}\theta_{k}a_{-k}\right), & \mbox{ if }
\mathcal{L}(\underline{a}) \,\,\mbox{is odd} \\
\bar{\psi}\left(\bar{\theta}_{0}+\sum_{k\geq1}\bar{\theta}_{k}a_{-k}\right), & \mbox{ if }
\mathcal{L}(\underline{a}) \,\, \mbox{is even.}
\end{array} \right.
\]
This model  has a discontinuity at
$-\underline{1}$. To see this, observe that
$P(1|(-1)_{-k}^{-1}\,\underline{1})$ takes value
$\psi\left(\theta_{0}-\sum_{i=1}^{k}\theta_{i}+\sum_{n\ge k+1}\theta_{n}\right)$ or
$\bar{\psi}\left(\bar{\theta}_{0}-\sum_{i=1}^{k}\bar{\theta}_{i}+\sum_{n\ge k+1}\bar{\theta}_{n}\right)$, according to 
$k$ being odd or even, and therefore does not converge in $k$, as $\psi(\theta_{0})\neq
\bar{\psi}(\bar{\theta}_{0})$. Nevertheless, using similar calculations as in Example \ref{ex:AR}, we observe that this kernel belongs to LC($\tau$), since $\alpha_{k}^{i}$ takes value $1-2\psi'(c)\sum_{j\ge k+1}|\theta_{j}|\wedge 1-2\bar{\psi}'(c)\sum_{j\ge k+1}|\bar{\theta_{j}}|$ and converges to $1$ as $k$ diverges. Since this quantity does not depend on $i$, it follows that $\alpha_{k}:=\inf_{i\ge1}\alpha^{i}_{k}$ converges to $1$ as well and therefore, $P$ belongs to ULC($\underline{\tau}$). Using Corollary \ref{coro:ulc} together with Proposition \ref{prop:skel1}, we conclude that  the rate at which $1-(\frac{1}{\epsilon}+1)(1-\alpha_{m})$ converges to zero controls the tail distribution of the coalescence time of the CFTP.}
\end{example}

\begin{example}\label{ex:eu}
\emph{The following example is inspired in \cite{gallo/2009}. We let $h:\mathbb{N}\rightarrow2\mathbb{N}+1$ be non-decreasing and unbounded and for any $v\in A^{\star}$ with $|v|$ odd, we let $\textrm{Maj}(v)$ denotes the symbol that most appears in $v$ (Maj stands for \emph{majority} here). Then, we put, for any $\underline{a}$ with $\mathcal{L}(\underline{a})=l\ge1$
\[
P(1|\underline{a})=\epsilon_{l}{\bf1}\left\{\textrm{Maj}(a^{-l-1}_{-l-h(l)})=-1\right\}+(1-\epsilon_{l}){\bf1}\left\{\textrm{Maj}(a^{-l-1}_{-l-h(l)})=1\right\}
\]
where $\{\epsilon_{l}\}_{l\ge0}$ is a $]\epsilon,1/2-\epsilon[$-valued sequence, with $0<\epsilon<1/4$. Put also $P(1|-\underline{1})\ge\epsilon$. Here also, we can compute
\[
\omega_{k}(-\underline{1})=1-\sup_{l,m\ge k}|\epsilon_{l}-\epsilon_{m}|,
\]
which will not converge to $1$ if and only if $\{\epsilon_{l}\}_{l\ge0}$ does not converge. We assume therefore that 
this is the case, and since for any $\underline{a}$ such that $\mathcal{L}(\underline{a})=l<+\infty$ we have $\omega_{k}(\underline{a})=1$ for any $k\ge l+h(l)+1$, we conclude that $P$ is discontinuous and belongs to SLC($\underline{\tau}$). Applying Corollary \ref{coro:slc} together with Proposition \ref{prop:skel1}, we conclude that the way $(\frac{1}{\epsilon}+1)\mathbb{P}(c_{\tau}^{0}>h^{-1}(m))=(\frac{1}{\epsilon}+1)(1-\epsilon)^{h^{-1}(m)}$ converges to $0$ controls the tail distribution of the coalescence time of the CFTP. For instance, when
\[
\limsup\frac{\log h(k)}{C_{\epsilon}k}<1\,,\,\,\,\,C_{\epsilon}=-\log(1-\epsilon),
\] 
we get, as obtained by Theorem 1 in \cite{gallo/2009}, that the coalescence time of the CFTP has summable tail. }

\end{example}

\section{Proof of Theorem \ref{theo2}} \label{sec:prova2}

\subsection{The update and the length function} Let us say, before going into any further details, that the update function we will use is the same (with some simple changes to make it suitable when $P$ is not necessarily continuous) as the one used by \cite{comets/fernandez/ferrari/2002}, and which underlies the works of \cite{gallo/2009} and \cite{desantis/piccioni/2012}. 
This update function is defined through  the partition of $[0,1[$  represented on Figure \ref{fig:partition}, where for any $a$ and $\underline{a}$, the intervals have length
\begin{align*}
&|I_{0}(a|\emptyset)|=\alpha_{-1}(a)\\
&|I_{k}(a|a_{-k}^{-1})|:=\inf_{\underline{z}}P(a|a_{-k}^{-1}\,\underline{z})-\inf_{\underline{z}}P(a|a_{-k+1}^{-1}\,\underline{z}),\,\,\,\,\forall k\geq1,\\
&|I_{\infty}(a|\underline{a})|:=P(a|\underline{a})-\lim_{k\rightarrow \infty}\inf_{\underline{z}}P(a|a_{-k}^{-1}\,\underline{z}), 
\end{align*}
The only difference  with the partition  used by \cite{comets/fernandez/ferrari/2002} (see Figure 1 therein),   is the addition of the intervals $|I_{\infty}(a|\underline{a})|$, due to the fact that when $P$ is not assumed to be continuous, $\omega_{k}(a|a_{-k}^{-1})$ may not converge to $1$ for some pasts.
\begin{figure}[htp]
\centering
\includegraphics{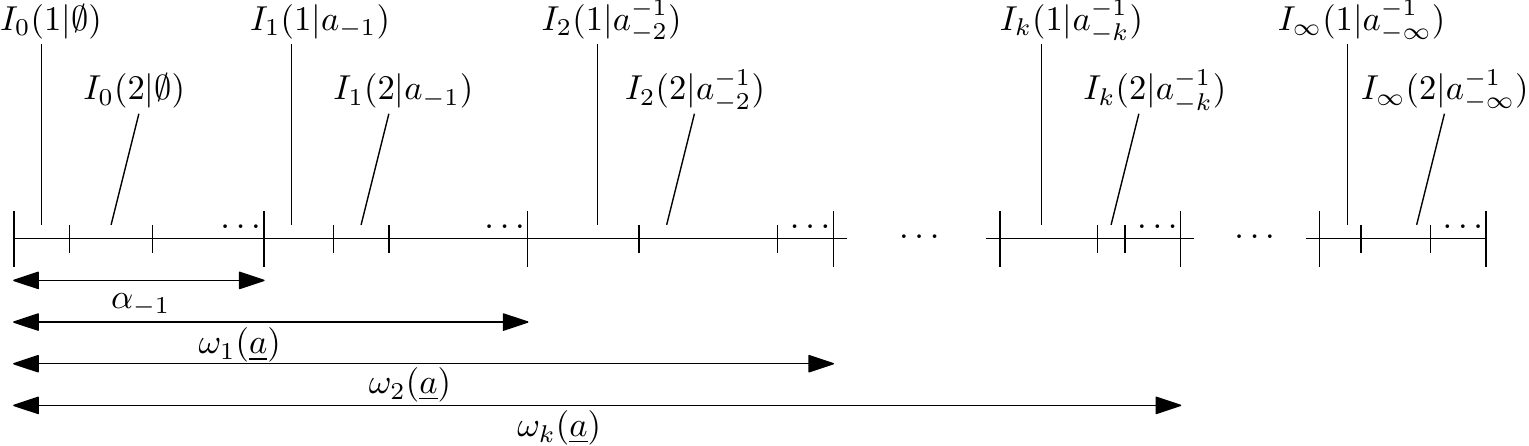}
\caption{Illustration of the partition related to some infinite past $\underline{a}$.}
\label{fig:partition}
\end{figure}

With this partition in hands, the update function is simple to define:
\[
F(U_{0},\underline{a}):=\sum_{a\in A}a{\bf1}\{U_{0}\in\cup_{k\ge0}\cup_{a\in A}I_{k}(a|a_{-k}^{-1})\}.
\]
Algorithmically, this update function is practical for the following reason. If we introduce the length function
\begin{equation}\label{eq:Lfunction}
L(U_{0},\underline{a}):=\sum_{k\ge0}k{\bf1}\{U_{0}\in\cup_{a\in A}I_{k}(a|a_{-k}^{-1})\}+\infty.{\bf 1}\{U_{0}\in \cup_{a\in A}I_{\infty}(a|\underline{a})\}
\end{equation}
we observe that, whenever $L(U_{0},\underline{a})\leq k<\infty$ (this occurs when $U_{0}\leq \omega_{k}(\underline{a})$) we have $F(U_{0},\underline{a})=F(U_{0},\underline{b}a_{-k}^{-1})$ for any $\underline{b}$. This means that the value of $F(U_{0},\underline{a})$ can be decided looking only at (at most) the $k$ last symbols of the infinite past $\underline{a}$.
In other words, once the algorithm constructs $k$ symbols $a_{-k}^{-1}$, from, say, times $n-k$ to time $n-1$, the construction of the next symbol is possible when $L(U_{n},\underline{a})\leq k$ because this event only depends on $a_{-k}^{-1}$. This is not only a ``practical'' advantage, but also a mathematical advantage to  prove that $\theta[0]$ (defined as \eqref{eq:theta} using this update function) is $\mathbb{P}$-a.s. finite. 

Let us consider, for any $-\infty<m\leq n\leq +\infty$
\[
\theta'[m,n]:=\max\{i\leq m:\,\textrm{for any }\underline{a}\,,\,\,L(\underline{a}F_{\{i,j-1\}}(\underline{a},U_{i}^{j-1}),U_{j})\leq j-i\,,\,j=i,\ldots,n\}
\]
with the convention that $\theta[m]:=\theta[m,m]$ and where $F_{\{m,n\}}(\underline{a},U_{m}^{n})$ denotes the whole constructed string based on the past $\underline{a}$, that is, for any $-\infty<m\leq n\leq +\infty$
\[
F_{\{m,n\}}(\underline{a},U_{m}^{n}):=F(\underline{a},U_{m})F_{[m,m+1]}(\underline{a},U_{m}^{m+1})\ldots F_{[m,n]}(\underline{a},U_{m}^{n}).
\]

\begin{lemma}\label{lem:3} $\theta'[0]\leq \theta[0]$.
\end{lemma}

\begin{proof} 
We will prove that  $\theta'[0]\in\{j\le 0:F_{\{j,0\}}(\underline{a},U_{j}^{0})\,\,\,\textrm{does not depend   on }\,\underline{a}\}.$  
Assume $\theta'[0]=-k>-\infty$. We proceed recursively.  First observe that $U_{-k}\in
  [0,\alpha_{-1}[$ since $L(U_{-k},\underline{a})=0$ for any
  $\underline{a}$, and therefore,
  $F_{\{-k,-k\}}(\underline{a},U_{-k})=F(\underline{a},U_{-k})$ is
  obtained independently of $\underline{a}$. Suppose (recursion
  hypothesis) that 
  for some $l\in\{1,\ldots,k-1\}$,
  the whole string
  $F_{\{-k,-l\}}(\underline{a},U_{-k}^{-l})$ has been constructed
  independently of $\underline{a}$. Since
  $L(\underline{a}F_{\{-k,-l\}}(\underline{a},U_{-k}^{-l}),U_{-l+1})\leq
  k-l+1$, it follows that the value of
  $F(\underline{a}F_{\{-k,-l\}}(\underline{a},U_{-k}^{-l}),U_{-l+1})$ can be obtained independently of $\underline{a}$, and thus,
  concatenating, we obtain that the whole string
  $F_{\{-k,-l+1\}}(\underline{a},U_{-k}^{-l+1})$ has been constructed
  independently of $\underline{a}$, establishing the recursion from
  time $-k$ to time $0$.
\end{proof}

\subsection{Definition of a block-rescaled coalescence time}

 Lemma \ref{lem:3} 
indicates that we can focus on $\theta'[0]$ instead of $\theta[0]$. However, the direct study of $\theta'[0]$ remains an intricate task, and our objective here is to introduce the coalescence time $\Lambda[0]$, defined by equation \eqref{eq:truerege}, which is easier to study. We will nevertheless need to introduce a sequence of technical definitions in order to get to its definition.
Let $\{\theta^{k}\}_{k\geq-1}$ be
the sequence  of r.v.'s defined  by $\theta^{-1}:=1$ and for any $k\geq 0$ 
\[
\theta^{k}=\bar{\theta}[\theta^{k-1}-1],
\]
and partition $-\mathbb{N}$ into disjoint blocks $\{B_{k}\}_{k\geq0}$ where $B_{k}=\{\theta^{k},\ldots,\theta^{k-1}-1\}$. 

We consider the sequence $\{\zeta_{l}\}_{l\in\mathbb{Z}}$ defined for any $i\in\mathbb{Z}$ by
\[
\zeta_{l}:={\bf1}\{U_{l}\ge\alpha_{-1}\}\sum_{k\ge0}k.{\bf1}\{U_{l}\in[\alpha_{k-1}^{c_{\tau}^{l-1}},\alpha_{k}^{c_{\tau}^{l-1}}[\},
\]
and define for any $k\ge0$
\begin{equation}\label{eq:L}
L_{k}:=\sup_{l\in B_{k}}\zeta_{l}.
\end{equation}

An important advantage of this block rescalling, that we will use later, is the fact that the sequence $\{L_{k}\}_{k\geq0}$ is i.i.d. This  follows from the definition of good coalescence time, which implies that for any $j\in B_{k}$, $k\ge0$, such that $U_{j}\ge \alpha_{-1}$, we have
\begin{equation}\label{eq:maxcont}
c_{\tau}^{j-1}=\sup_{{\bf a}}|c_{\tau}(Y({\bf a})_{\theta^{k}}^{j-1})|\leq j-\theta^{k},
\end{equation}
and therefore, $\zeta_{l}$ for $i\in B_{k}$ can be determined using the array $\{U_{j}\}_{j\in B_{k}}$.

Now, introduce the random variable
\begin{equation}\label{eq:regeblocks}
\Theta[0]=\Theta[0]({\bf U}):=\sup\{n\leq 0:L_{i}\leq i-n\,,\,\,i=n,\ldots,0\}
\end{equation}
which will play the role of a coalescence time of $B_0$, in
the block rescaled sequence.
We finally define the main
random variable of interest for the proof of the theorem 
\begin{equation}\label{eq:truerege}
\Lambda[0]=\Lambda[0]({\bf U}):=-\sum_{i=0}^{-\Theta[0]}|B_{i}|.
\end{equation}
We now state two important lemma. Lemma \ref{lem:5} ensures that $\Lambda[0]$ is indeed a coalescence time for time $0$, and Lemma \ref{key-lemma} gives informations on its tail distribution. 

\begin{lemma}\label{lem:5}
$\Lambda[0]\leq \theta'[0]$.
\end{lemma}
\begin{proof}
We will prove that $\Lambda[0]$ belongs to the set of coalescence times $\{i\leq 0:\,\textrm{for any }\underline{a}\,,\,\,L(\underline{a}F_{\{i,j-1\}}(\underline{a},U_{i}^{j-1}),U_{j})\leq j-i\,,\,j=i,\ldots,0\}$. We proceed   in three steps. 

\hspace{0.5cm}{\bf Step 1.} Observe that, for any $k\ge0$ and any $i\in B_{k}$, $
\zeta'_{i}=i-\theta^{k}+{\bf 1}\{U_{i}\ge \alpha_{-1}\}\left(\sum_{j=1}^{L_{k}}|B_{k+j}|\right),\,\,\,\,\forall i\in B_{k} \,\textrm{ and }\, k\ge0$, we can rewrite  $\Lambda[0]$ as $\max\{i\leq m:\,\zeta'_{j}\leq j-i\,,\,j=i,\ldots,n\}$

\hspace{0.5cm}{\bf Step 2.} Since (i)  $|B_{k}|\ge1$, $\forall k\ge0$, (ii) for any $i\in B_{k}$ we have $L_{k}\ge \zeta_{i}$ and (iii) for any $i\in B_{k}$ such that $U_{i}\ge \alpha_{-1}$, we have
\begin{equation}\label{eq:maxcont}
c_{\tau}^{i-1}=\sup_{{\bf a}}|c_{\tau}(Y({\bf a})_{\bar{\theta}[i]}^{i-1})|\leq i-\theta^{k},
\end{equation}
it follows that $\zeta_{i}'\ge \xi_{i}:={\bf 1}\{U_{i}\ge \alpha_{-1}\}c_{\tau}^{i-1}+\zeta_{i}$. 

\hspace{0.5cm}{\bf Step 3.} 
We will prove that,  for any realization ${\bf u}$ of the process ${\bf U}$ such that $\xi_{l}\leq k$, we have $L(\underline{a}F_{\{l-k,l-1\}}(\underline{a},u_{l-k}^{l-1}),u_{l})\leq k$ for any $\underline{a}\in A^{-\mathbb{N}}$. 
This is trivial if $u_{l}\leq \alpha_{-1}$, thus we assume that this is not the case. We have the following sequence of inequalities.
First,  by  \eqref{eq:ctau}, supposing  $c_{\tau}^{l-1}=m$
\begin{equation}\label{eq:inequaliti}
|c_{\tau}(\underline{a}F_{\{l-k,l-1\}}(\underline{a},u_{l-k}^{l-1}))|\leq c_{\tau}^{l-1}=m.
\end{equation}
 This  implies, together with \eqref{eq:alphaki}, that, for any $i\ge0$, 
\[
\alpha_{i}^{m}= \alpha_{i}^{c_{\tau}^{l-1}}\leq \alpha_{i}^{|c_{\tau}(\underline{a}F_{\{l-k,l-1\}}(\underline{a},u_{l-k}^{l-1}))|}\leq\alpha_{i}^{c_{\tau}(\underline{a}F_{\{l-k,l-1\}}(\underline{a},u_{l-k}^{l-1}))},
\]
and by \eqref{condition1} and \eqref{eq:pointcontinuity}, we have 
\[
\alpha_{i}^{c_{\tau}(\underline{a}F_{\{l-k,l-1\}}(\underline{a},u_{l-k}^{l-1}))}\leq \omega_{i+|c_{\tau}(\underline{a}F_{\{l-k,l-1\}}(\underline{a},u_{l-k}^{l-1}))|}(\underline{a}F_{\{l-k,l-1\}}(\underline{a},u_{l-k}^{l-1})).
\]
Since $\xi_{l}\leq k$, we have, using the above inequalities
\[
u_{l}\leq \alpha_{k-m}^{m}\leq  \omega_{k-m+|c_{\tau}(\underline{a}F_{\{l-k,l-1\}}(\underline{a},u_{l-k}^{l-1}))|}(\underline{a}F_{\{l-k,l-1\}}(\underline{a},u_{l-k}^{l-1})).
\] 
According to the partition defining $F$ and $L$, this implies that the length function  $L(\underline{a}F_{\{l-k,l-1\}}(\underline{a},u_{l-k}^{l-1}),u_{l})\leq k-m+|c_{\tau}(\underline{a}F_{\{l-k,l-1\}}(\underline{a},u_{l-k}^{l-1}))|$, and using one more time \eqref{eq:inequaliti} concludes the proof of Step 3.\\

We complete the proof of the lemma  using steps 1, 2 and 3.

\end{proof}

\begin{lemma}[Key-Lemma]\label{key-lemma}
Consider a kernel $P$ belonging to LC($\tau$) with good probabilistic skeleton $(\tau,p)$, and  let $\bar{\theta}[0]$ be the corresponding good coalescence time. 
\begin{enumerate}[(i)]
\item  If $\sum_{k\geq1}\prod_{j=0}^{k-1}\P(L_{0}\leq i)=+\infty$, then $\Lambda[0]$ is $\mathbb{P}$-a.s. finite.
\item If $\bar{\theta}[0]$ has summable tail and $\sum_{k\ge0}\P(L_{0}>i)<+\infty$, then $\Lambda[0]$ has summable tail.
\item If $\bar{\theta}[0]$ has exponential tail and
  $\{\P(L_{0}>i)\}_{k\geq0}$ decays exponentially fast to zero, then
  $\Lambda[0]$ has exponential tail.
\end{enumerate}
In particular, since $\Lambda[0]\leq \theta'[0]\leq\theta[0]$, it follows that in each of these regimes, the same conclusion hold for all these coalescence times.
\end{lemma}

\begin{proof}
In order to control the tail distribution of $\Lambda[0]$ (see \eqref{eq:truerege}), we first need to control the tail distribution of $\Theta[0]$.
Observe that $\Theta[0]$ is defined over the block rescaled sequence using the i.i.d. sequence of r.v.'s $\{L_{i}\}_{i\geq0}$ exactly as the coalescence time of site $0$ was defined in \cite{comets/fernandez/ferrari/2002} (where it is denoted by $\tau[0]$). Indeed, since the $L_{i}$'s are i.i.d., and since $\mathbb{P}(L_{0}=0)\geq\mathbb{P}(|B_{0}|=1)=\alpha_{-1}>0$, we can invoke Theorem 4.1 item (iv) together with Proposition 5.1 in \cite{comets/fernandez/ferrari/2002} and obtain the following assertions
\begin{description}
\item (\ref{key-lemma}.1) if $\sum_{k\geq1}\prod_{i=0}^{k-1}\mathbb{P}(L_{0}\leq i)=+\infty$, then $\Theta[0]$ is a.s. finite,
\item (\ref{key-lemma}.2) if $\{\mathbb{P}(L_{0}>k)\}_{k\geq0}$ is summable, then $\Theta[0]$ has summable tail,
\item (\ref{key-lemma}.3) if $\{\mathbb{P}(L_{0}>k)\}_{k\geq0}$ decays exponentially fast to $0$, then $\Theta[0]$ has exponential tail.
\end{description} 
Coming back to the definition \eqref{eq:truerege} of $\Lambda[0]$, we
now prove items (i), (ii) and (iii) of the lemma using respectively items
(\ref{key-lemma}.1), (\ref{key-lemma}.2) and (\ref{key-lemma}.3) we
just stated.  Item (i) is direct since the sum of an a.s. finite
number of random variables which are a.s. finite is a.s. finite. For
item (ii),we observe that
\[
\sum_{i=0}^{n-1}|B_{i}|-n\mathbb{E}|B_{0}|
\]
is a martingale with respect to the filtration  $\mathcal{F}((L_{0},|B_{0}|),\ldots,(L_{-i},|B_{-i}|):i\geq0)$. 
Moreover, $\Theta[0]$ is a stopping time with respect to the same filtration, this follows directly from the definitions of $\Theta[0]$. Thus, by the Optional Stopping Theorem
\[
\mathbb{E}\Lambda[0]=\mathbb{E}\left(\sum_{i=0}^{-\Theta[0]}|B_{i}|\right)=\mathbb{E}|B_{0}|.\mathbb{E}|\Theta[0]|,
\]
which is finite by item (\ref{key-lemma}.2) above, in the conditions of item (ii) of the key-lemma. For the proof of item (iii), we use the proof of item (iv) of Lemma 14 in \cite{harvey/holroyd/peres/romik/2007}. Let $\rho:=(2\mathbb{E}|B_{0}|)^{-1}$ and compute 
\[
\mathbb{P}(\Lambda[0]<-n)=\sum_{i\geq0}\mathbb{P}(\Theta[0]=-i,\sum_{j=0}^{i}|B_{j}|>n)
\]
\[
\leq\sum_{i=0}^{\lfloor\rho.n\rfloor}\mathbb{P}\left(\Theta[0]=-i,\sum_{j=0}^{i}|B_{j}|>n\right)+\sum_{i\geq\lfloor\rho.n\rfloor+1}\mathbb{P}(\Theta[0]=-i)
\]
\begin{equation}\label{eq:paraMPRF}
\leq\lfloor\rho.n\rfloor\mathbb{P}\left(\left|\sum_{j=0}^{\lfloor\rho.n\rfloor}|B_{j}|-\lfloor\rho.n\rfloor\mathbb{E}|B_{0}|\right|>n/2\right)+\mathbb{P}(\Theta[0]<-\lfloor\rho.n\rfloor)
\end{equation}
Because $|B_{0}|$ has exponential tail, a standard result of large deviation (see for example Corollary 27.1 in \cite{kallenberg/2002}) shows that the first term in the last line decays exponentially in $n$. This proves item (iii) using item (\ref{key-lemma}.3) above. 
\end{proof}

\subsection{Proof of Theorem \ref{theo2}}
Using Lemmas \ref{lem:3}, \ref{lem:5} and \ref{key-lemma}, it remains to prove that  $A_{m}\leq \mathbb{P}(L_{0}\leq m)$, $m\ge0$.
For any $i\ge0$
\begin{align}
\mathbb{P}(L_{0}>i)=&\mathbb{P}(\sup_{j\in B_{0}}\zeta_{j} >i)=\mathbb{P}(\sum_{j=\bar{\theta}[0]}^{0}{\bf1}\{\zeta_{j} >i\}\ge1)\leq\mathbb{E}(\sum_{j=\bar{\theta}[0]}^{0}{\bf1}\{\zeta_{j}>i\}).
\end{align}
We have all the ingredient for applying the Wald inequality:
\begin{itemize}
\item By translation invariance, we have $\mathbb{E}\zeta_{l}=\mathbb{E}\zeta_{0}$ for any $l\in\mathbb{Z}$.
\item By our assumptions, $\mathbb{E}|\bar{\theta}[0]|<\infty$.
\item Finally, for any $n\ge1$,
\begin{align*}
\mathbb{E}({\bf1}\{\zeta_{-n}>i\}.{\bf1}\{\bar{\theta}[0]\leq -n\})&=\mathbb{E}{\bf1}\{\zeta_{-n}>i\}-\mathbb{E}({\bf1}\{\zeta_{-n}>i\}.{\bf1}\{\bar{\theta}[0]> -n\})\\
&=\mathbb{E}{\bf1}\{\zeta_{0}>i\}-\mathbb{E}{\bf1}\{\zeta_{-n}>i\}.\mathbb{E}({\bf1}\{\bar{\theta}[0]> -n\})\\
&=\mathbb{E}{\bf1}\{\zeta_{0}>i\}.\left[1-\mathbb{P}(\bar{\theta}[0]> -n)\right]\\
&=\mathbb{P}(\zeta_{0}>i).\mathbb{P}(\bar{\theta}[0]\leq-n)
\end{align*}
where, for the second line, we used the fact that $\{\bar{\theta}[0]> -n\}$ is measurable with respect to $\mathcal{F}(U_{-n+1}^{0})$, while ${\bf1}\{\zeta_{-n}>i\}$ is $\mathcal{F}(U_{-\infty}^{-n})$-measurable. 
\end{itemize}

Thanks to all these facts, we can use Wald's equality, and obtain
\[
\mathbb{P}(L_{0}>i)\leq \mathbb{E}(|\bar{\theta}[0]|+1)\mathbb{P}(\zeta_{0}>i)=\mathbb{E}(|\bar{\theta}[0]|+1)\mathbb{P}(U_{0}\ge \alpha_{i}^{c_{\tau}^{-1}}),
\]
and since for any $i\ge0$ we have $\mathbb{P}(L_{0}\leq i)\geq \mathbb{P}(L_{0}=0)\ge \alpha_{-1}$, the proof of the theorem is concluded using $A_{k}:=\left\{1-(\mathbb{E}|\bar{\theta}[0]|+1)\mathbb{P}(U_{0}>\alpha_{k}^{c_{\tau}^{-1}})\right\}\vee \alpha_{-1}$.$\hspace{2.5cm}\square$

\section{Relaxing the local continuity
  assumption}\label{sec:extension}

In this section we propose an extension of the notion of local
continuity. 
Local continuity
corresponds to assume that there exists a stopping time for the
reversed-time process, beyond which the decay of the dependence on the
past occurs uniformly. Removing this assumption means that no stopping
time can tell whether or not the past we consider is a continuity
point for $P$. This  idea is now formalized in the following definition.

\begin{defi}\label{def:extendedLC}
We will say that a kernel $P$ belongs to the class of \emph{extended locally continuous} with respect to some skeleton $\tau$ if  for any $v\in\leftexp{<\infty}{\tau}$
\begin{equation}\label{def:Enew}
\bar{\alpha}_{k}^{v}:=\inf_{v_{1}\in \tau}\inf_{v_{2}\in \tau}\ldots\inf_{v_{k}\in\tau}\sum_{a\in A}\inf_{\underline{z}}P(a|v\,v_{1}\,v_{2}\ldots\,v_{k}\,\underline{z})\stackrel{k\rightarrow+\infty}{\longrightarrow}1.
\end{equation}
We will denote this class by   \emph{extLC($\tau$)}. The \emph{probabilistic skeleton (p.s.) of $P$}  is the pair $(\tau,p)$ where $p:=\{p(a|v)\}_{a\in A,\,v\in\leftexp{<\infty}{\tau}}$,
\begin{equation}\label{eq:p}
p(a|v):=\inf_{\underline{z}}P(a|v\underline{z})
\end{equation}
and $p(a|v)=P(a|v)$ for any $v\in \leftexp{\infty}{\tau}$.

\end{defi}

\begin{theo}\label{theo2.5}
Assume that $P$ belongs to extLC($\tau$) with good probabilistic skeleton. Assume furthermore that the good coalescence time $\bar{\theta}[0]$ satisfies that $c_{\tau}^{j}\leq j-\bar{\theta}[0]$ for any $j\in\{\bar{\theta}[0],\ldots,0\}$.
For  any $k\ge0$, denote
$$\bar{A}_{k}:=\left\{1-(\mathbb{E}|\bar{\theta}[0]|+1)\mathbb{P}(U_{0}>\bar{\alpha}_{k}^{c_{\tau}^{-1}})\right\}\vee \alpha_{-1},\,\,\,\,k\ge0.$$
Then, we can construct for $P$, an update
  function $F$ and a corresponding coalescence time $\theta$ such
  that 
\begin{enumerate}[(i)]
\item  If $\sum_{k\geq1}\prod_{j=0}^{k-1}\bar{A}_{k}=+\infty$, then $\theta[0]$ is $\mathbb{P}$-a.s. finite.
\item If  $\sum_{k\ge0}(1-\bar{A}_{k})<+\infty$, then $\theta[0]$ has summable tail.
\item If $\bar{\theta}[0]$ has exponential tail and $\{1-\bar{A}_{k}\}_{k\geq0}$ decays exponentially fast to zero, then $\theta[0]$ has exponential tail.
\end{enumerate}
In particular, in each of these regimes, the CFTP with update function $F$ is feasible.
\end{theo}

Observe that the assumption that $c_{\tau}^{j}\leq j-\bar{\theta}[0]$
for any $j\in\{\bar{\theta}[0],\ldots,0\}$ is a bit stronger than
simply assuming that $\bar{\theta}[0]$ is good since this
later assumption only implies that $c_{\tau}^{j}\leq
j-\bar{\theta}[0]$ for the time indexes
$j\in\{\bar{\theta}[0],\ldots,-1\}$ such that
$Y_{j+1}=\star$. Nevertheless, for skeletons having
terminal context $w$, $\bar{\theta}[0]:=\max\{i\leq
-|w|+1:Y_{i}^{i+|w|-1}=w\}$ used in Proposition \ref{prop:skel1}
satisfies this stronger assumption. For skeletons
$\tau^{\sigma}$ as considered in Proposition \ref{prop:skel2}, the
good coalescence time
$\bar{\theta}[0]:=\min\{i\geq0:\frac{1}{i+1}\sum_{j=-i}^{0}Y({\bf
  0})_{j}\ge\sigma\}$ also satisfies the stronger assumption. We now give the proof
  of this theorem, and then give two examples on $A=\{-1,+1\}$, and an application for creation of
  new skeleton that can be used for Theorem \ref{theo2}.\\

\proof The proof of this theorem follows exactly the same steps as the
proof of Theorem \ref{theo2}. To avoid repetition, we just outline the
key observation and leave the proof to the reader. Observe that, if at
time $i$, the random variable $\bar{L}_{i}$ takes value $l$, this
means that we have to look at a portion of the past preceding $B_{i}$
which contains $l$ concatenated contexts of $\tau$. But in the
conditions of the theorem, the blocks themselves are concatenation of
contexts of $\tau$, then, it is still true that we have at least $l$
concatenated contexts contained in the $l$ blocks preceding $B_{i}$.

Therefore, if $\alpha_{k}=\bar{\alpha}_{k}$, the number of blocks
involved in both procedures is the same. Thus, the random variable
$\bar{\Lambda}[0]$ obtained with the new decomposition has essentially
the same tail distribution as $\Lambda[0]$.\\

\begin{example}\emph{For any $\underline{a}\in A^{-\mathbb{N}}$, let 
\[
\mathcal{K}(\underline{a}):=\inf\{i\ge1:a_{-k}=a_{-k-1},\,k\ge i\}
\]
with the convention that $\mathcal{K}(\underline{a})=+\infty$ if the
set is empty. Now let  $P$ be defined by
$$P(1|\underline{a})=\epsilon+\frac{1}{f[\mathcal{K}(\underline{a})]},$$
where $f$ is an unbounded increasing integer valued function satisfying $f(1)\ge\frac{1}{1-2\epsilon}$. This way we get 
$\inf_{\underline{a}}P(-1|\underline{a})=1-\epsilon-\frac{1}{f(1)}\ge\epsilon$.
We
will show that this kernel has discontinuities at every point having
either finitely many $-1$'s or finitely many $1$'s. Let $\mathcal{S}$
denote the set of such points and take $\underline{a}\in\mathcal{S}$. We
have, for any $\underline{z}\in A^{-\mathbb{N}}\setminus\mathcal{S}$,
$P(1|a_{-k}^{-1}\underline{z})=\epsilon$ for any $k$ and
therefore, does not converge to
$\epsilon+\frac{1}{\mathcal{K}(\underline{a})}$. On the other hand,
the kernel is continuous at any $\underline{a}\in
A^{-\mathbb{N}}\setminus\mathcal{S}$, since
$P(1|a_{-k}^{-1}\underline{z})=\epsilon+\frac{1}{\mathcal{K}(a_{-k}^{-1}\underline{z})}$ 
which always converge to $\epsilon$ because
$\mathcal{K}(a_{-k}^{-1}\underline{z})$ converges to $+\infty$ (or
equals infinity when $\underline{z}\in
A^{-\mathbb{N}}\setminus\mathcal{S}$).}

\emph{The main point is that the integer
$\mathcal{K}(\underline{a})$ can never be checked looking at a finite
portion of the past, while both, $\mathcal{L}(\underline{a})$ and
$T_{\sigma}(\underline{a})$ (used to define the kernels of Section \ref{sec:ex}) can be checked looking at a finite portion
of the past of $\underline{a}$. This means that no matter how much we
know of $\underline{a}$, we never make sure that it is indeed
continuous past for $P$.  We now explain why this new kernel satisfies the conditions of Theorem \ref{theo2.5}. First, we
observe that it belongs to extLC($\tilde\tau$) with 
\[
\tilde\tau=\{-\underline{1}\}\cup\bigcup_{i\ge0}\{1(-1)^{i}\}\cup\{\underline{1}\}\cup\bigcup_{i\ge0}\{-11^{i}\}
\]
since for any $k\ge0$ and any set $\{v,v_{1},\ldots,v_{k}\}$ of elements of $\tilde\tau$, we have
\begin{align*}
\sum_{a\in A}\inf_{\underline{z}}P(a|v\,v_{1}\,v_{2}\ldots\,v_{k}\,\underline{z})&=\inf_{\underline{z}}P(1|v\,v_{1}\,v_{2}\ldots\,v_{k}\,\underline{z})+1-\sup_{\underline{z}}P(1|v\,v_{1}\,v_{2}\ldots\,v_{k}\,\underline{z})\\
&=1-\frac{1}{f[|v|+\sum_{i=1}^{k}|v_{i}|]}
\end{align*}
and therefore, $\inf_{v\in\tilde{\tau}}\bar{\alpha}^{v}_{k}=1-\frac{1}{f(2k+2)}$ which converges to $1$ as $k$ diverges. Also, observe that the p.s. $(\tilde{\tau},p)$ satisfies $p(a|v)\ge\epsilon$ for any $a\in A$ and $v\in\tilde{\tau}$, and therefore, fits the conditions of Proposition \ref{prop:skel1}, because $\tilde{\tau}$ has $(-1)1$ and $1(-1)$ as terminal strings. It follows that $\mathbb{E}|\bar{\theta}[0]|\leq 1/\epsilon^{2}$
and thus $$\bar{A}_{k}\ge 1-(1+\frac{1}{\epsilon^{2}})\frac{1}{f(2k+2)}.$$
According to the function $f$ we choose, we can have the three regimes of CFTP specified by Theorem \ref{theo2.5}. }

\end{example}

\begin{example}\label{ex:piccioni1}\emph{This example is taken from  \cite{desantis/piccioni/2012} (Example 1 therein). It is  defined using a sequence of real numbers $\{\theta_{n}\}_{n\ge1}$ such that $\sum_{k\ge1}|\theta_{k}|<1/2$,  a continuous function $f:\mathbb{R}^{+}\rightarrow[0,1]$ decreasing to $0$, and the  quantity
\[
S_{k}(a_{-k}^{-1})=\sum_{i=1}^{k-1}{\bf 1}\{a_{-i}\neq a_{-i-1}\}, 
\]
which counts the number of changes of signal in $a_{-k}^{-1}$. For any $\underline{a}\in A^{-\mathbb{N}}$, let
\[
P(1|\underline{a})=1/2+\sum_{k\ge1}\theta_{k}a_{-k}f(S_{k}(a_{-k}^{-1})).
\]
 Observe that this kernel is somewhat similar to the AR kernel (introduced in Section \ref{sec:CFFexample}) with $\psi=\textrm{Id}$ and $\theta_{0}=1/2$, with the difference that each occurrences of signal changes in the past reduces the dependence due to the multiplicative term in the sum. The same calculation as in the case of AR processes yields
 \[
 \omega_{k}(\underline{a})=1-2f(\beta S_{k}(a_{-k}^{-1}))\sum_{i\ge k+1}|\theta_{i}|.
 \]
This kernel is therefore continuous as for the AR process, with the difference that instead of a multiplicative term $\psi'(c(a_{-k}^{-1}))$ which is bounded away from $0$ and $+\infty$ (see \eqref{eq:1111}),   we now have a term $f(\beta S_{k}(a_{-k}^{-1}))$ which goes to zero for the pasts $\underline{a}$ having infinitely many changes of sign. In other words, these pasts have a faster continuity rate. Uniform continuity may not be useful because it amounts to take into account only the sequence $\{\sum_{i\ge k+1}|\theta_{i}|\}_{k\ge0}$ which may converge too slow to zero. This kernel is also an example in which the notion of  local continuity of Theorem \ref{theo2} cannot be used, since we cannot make sure whether a given past has a finite or infinite number of sign changes looking only at a finite portion. However, it  belongs to extLC($\tilde\tau$), since, for any $k\ge0$ and any set $\{v,v_{1},\ldots,v_{k}\}$ of elements of $\tilde\tau$, denoting $l=|v|+\sum_{i=1}^{k}|v_{i}|$
\begin{align*}
\sum_{a\in A}\inf_{\underline{z}}P(a|v\,v_{1}\,v_{2}\ldots\,v_{k}\,\underline{z})&=\inf_{\underline{z}}P(1|v\,v_{1}\,v_{2}\ldots\,v_{k}\,\underline{z})+1-\sup_{\underline{z}}P(1|v\,v_{1}\,v_{2}\ldots\,v_{k}\,\underline{z})\\
&=1-2f(\beta S_{l}(a_{-l}^{-1}))\sum_{i\ge l+1}|\theta_{i}|.
\end{align*}
Notice that we have $S_{l}(a_{-l}^{-1}))\ge k$, since there is at least $k$ sign changes in the concatenated string $v_{k}\ldots v_{1}v$. Thus, we obtain
\[
\inf_{v}\bar{\alpha}^{v}_{k}\ge 1-2f(\beta k)\sum_{i\ge 2k+3}|\theta_{i}|
\]
and the same calculations as in the preceding example yields
$$\bar{A}_{k}\ge 1-(1+\frac{1}{\epsilon^{2}})2f(\beta k)\sum_{i\ge 2k+3}|\theta_{i}|.$$
Here also, according to the function $f$ we choose, we can have the three regimes of CFTP specified by Theorem \ref{theo2.5}. For instance, the special case of $f(x)=e^{-\beta x}$, $\beta>0, x>0$ yields an exponential tail for the coalescence time of the CFTP (regime (iii) of our theorem).
}\end{example}

\paragraph{{\bf Application of Theorem \ref{theo2.5}}} By analogy with Definition \ref{def:slc}, we can define the class of kernels that satisfy the \emph{extended strong local continuity} with respect to some skeleton $\tau$ (let us denote extSLC($\tau$)). These are such that for any $v\in\leftexp{<\infty}{\tau}$, there exists a positive integer $h(v)$ such that
for any $k\geq h(v)$, $\bar{\alpha}_{k}^{v}=1$. The interesting point is that such kernels, when they satisfy the conditions of items (ii) or (iii) of Theorem \ref{theo2.5}, are good p.s.'s themselves. 

In other words, we have somehow a self-feeding argument, which allows us to construct more complicated good p.s.'s from simpler ones, using Theorem \ref{theo2.5}. And if we can show that $P$ is a good p.s., it can be used as such in Theorem \ref{theo2}.

We now explain why the kernels $P$ belonging to extSLC($\tau$) and satisfying the assumptions of Theorem \ref{theo2.5} are indeed good (see Definition \ref{def:good}). First, we observe
\[
\bar{\zeta}_{l}:={\bf1}\{U_{l}\ge\alpha_{-1}\}\sum_{k\ge0}k.{\bf1}\{U_{l}\in[\bar{\alpha}_{k-1}^{c_{\tau}^{l-1}},\bar{\alpha}_{k}^{c_{\tau}^{l-1}}[\}
\leq {\bf1}\{U_{l}\ge\alpha_{-1}\}h(c^{l-1}_{\tau})=:\bar{\bar{\zeta}}_{l}.
\]
This inequality means that we can use $\{\bar{\bar{\zeta}}_{l}\}_{l\in\mathbb{Z}}$
to obtain another coalescent time $\bar{\bar{\Lambda}}[0]$, defined in
the same way $\bar{\Lambda}[0]$ is defined using $\bar{\zeta}[0]$
(that is, just as we did in the proof of Theorem
\ref{theo2}). Clearly, we will have $\bar{\bar{\Lambda}}[0]\leq
\bar{\Lambda}[0]$, moreover  $\bar{\bar{\Lambda}}[0]$ is a good   coalescence time (see Definition \ref{def:good})
because
$\bar{\bar{\zeta}}_{l}$ is
$\mathcal{F}(Y_{-\infty}^{l})$-measurable. The fact that $P$ is in fact a good
p.s. follows now from the fact that
$\bar{\bar{\Lambda}}[0]$ has finite expectation under the assumptions of items (ii) and (iii) Theorem
\ref{theo2.5}.

\section{Concluding remarks}\label{sec:conclu}

There are several results that follow from {the
  existence of a CFTP scheme and} the regenerative structures.  Among them, {bounds for the
  $\bar{d}$-distance}, {rate of decay of correlations}, concentration inequalities and Functional
Central Limit Theorem.

\begin{itemize}

\item {\bf  Bounds for the $\bar{d}$-distance.} {Given a finite sample, it is natural to use a Markov
  approximation whose transition probabilities can be estimated from a
  sample of the infinite-order chain. A natural candidate would be the
  canonical $k^{\textrm{th}}$-order approximation, which is obtained
  by cutting off the memory after $k$ steps. Bounds on the
  $\bar{d}$-distance can be used to characterize the rate of
  convergence of estimators for stationary processes that can be
  approximated by $k$-steps Markov chains (see
  \cite{csiszar/talata/2010}). \cite{gallo/lerasle/takahashi/2011} use our perfect
  simulation scheme to derive new bounds for the $\bar{d}$-distance
  between the original chain and its canonical $k$-steps Markov
  approximation.}
\end{itemize}

\begin{itemize}
\item {\bf Loss of memory and decay of correlations.}
{CFTP  allows to directly obtain explicit upper bounds for the speed of the loss of memory of the chain, as it has been showed in \cite{comets/fernandez/ferrari/2002}. On the other hand, it is known that the decay of correlations is bounded above by this speed (see Remark 6.2.2 of 
\cite{maillard/2007} 
for instance). Roughly speaking, this means that both, decay of correlations and speed of loss of memory are controlled by the tail distribution of the coalescence time of the CFTP. 
}

\item{\bf Concentration of measures.}
Another direct application of CFTP is the following result on concentration of measures, proved in \cite{gallo/takahashi/2011}. 
\begin{prop} \label{prop:cftpconcentration}
Let ${\bf X}$ be a process that can be simulated by a CFTP algorithm with a coalescence time $\theta$. If $\E[\theta] < \infty$, then for all $\epsilon > 0$ and all functions $f:A^n \to \R$ we have
\begin{equation*} \label{eq:gaussianconcentration}
\P\left(\left|f(X^{n}_{1}) - \E[f(X^n_1)]\right| > \epsilon \right) \leq 2\exp\left\{{-\frac{2\epsilon^2}{(1+\E[\theta])^2
\|\delta f \|^2_{\ell_2(\N)}}} \right\}.
\end{equation*}
\end{prop}

\item {\bf Functional Central Limit Theorem.}
Under assumptions ensuring that the coalescence time of time $0$ has summable tail  (basically the number of steps that have to be performed by the algorithm in order to construct the stationary chain at time $0$), the constructed chain has a regeneration scheme. Such structure has been already observed under stronger assumptions using continuity, by for example, \cite{lalley/1986}, but also more recently in \cite{comets/fernandez/ferrari/2002} and \cite{gallo/2009}. It is worth mentioning that using this fact, a Functional Central Limit Theorem could be derived for our chains, as it has been done by \cite{maillard/schopfer/2008} under the continuity assumption. This is because, looking at their proof, we observe that it uses the regeneration property of the measure, and not the form of the conditional probabilities of the chain. 

\end{itemize}

\paragraph{Acknowledgments}
This work is part of USP Project ``Mathematics, computation, language
and the brain''.  SG was supported by a FAPESP fellowship (grant 2009/09809-1). 
  NG is supported by CNPq grants 475504/2008-9, 302755/2010-1 and
  476764/2010-6. We thank the anonymous referees for valuable remarks
  that greatly improved the presentation of this work. 

\bibliographystyle{jtbnew}
\bibliography{sandro_bibli}

\end{document}